\documentclass[12pt,reqno]{amsart}
\usepackage[a4paper,lmargin=3cm,rmargin=3cm,tmargin=3.5cm,bmargin=3.5cm]{geometry}
\usepackage[centertags]{amsmath}
\usepackage[english]{babel}
\usepackage{amsfonts}
\usepackage{amssymb}
\usepackage{amsthm,mathtools}
\usepackage{mathrsfs}
\usepackage{amsrefs}
\usepackage{pifont}
\usepackage[dvipsnames]{xcolor}
\usepackage{graphicx}
\usepackage{comment}
\usepackage{float}

\usepackage{hyperref}
\hypersetup{colorlinks=true, linkcolor=MidnightBlue, citecolor=teal, urlcolor=cyan}

\usepackage{tikz-cd}
\usepackage[T1]{fontenc}
\usetikzlibrary{patterns}
\usetikzlibrary{decorations.markings}
\tikzset{
  tachado/.style={
    decoration={
      markings,
      mark=at position 0.5 with {\node[ font=\sffamily] {X};}
    },
    postaction={decorate}
  }
}

\usepackage[shortlabels]{enumitem} 
\usepackage{bm}

\setlist[enumerate,1]{label=\textup{(\alph*)}}
\setlist[enumerate]{leftmargin=*, widest=ii}
\usepackage{pifont}

\newcommand{\ext}{\operatorname{ext}}
\newcommand{\str}{\operatorname{str-exp}}
\newcommand{\wexp}{w^*\operatorname{-exp}}
\newcommand{\wstr}{w^*\operatorname{-str-exp}}
\newcommand{\conv}{\operatorname{conv}}
\newcommand{\cconv}{\overline{\operatorname{conv}}}

\DeclareMathOperator{\inte}{int}
\DeclareMathOperator{\ainte}{a-int}
\DeclareMathOperator{\supp}{supp}

\DeclareMathOperator{\sgn}{sgn}
\DeclareMathOperator{\spann}{span}

\newcommand{\e}{\varepsilon}
\newcommand{\N}{\mathbb{N}}
\newcommand{\Z}{\mathbb{Z}}
\newcommand{\R}{\mathbb{R}}
\newcommand{\cal}[1]{\mathcal{#1}}
\newcommand{\B}{\mathcal{B}}
\newcommand{\D}{\mathcal{D}}

\newcommand{\cut}{\mathord{\upharpoonright}}
\newcommand{\vertt}[1]{{\left\vert\kern-0.25ex\left\vert\kern-0.25ex\left\vert #1 \right\vert\kern-0.25ex\right\vert\kern-0.25ex\right\vert}}
\newcommand{\n}{\left\Vert\cdot\right\Vert}
\newcommand{\nn}{\vertt{\cdot}}
\newcommand{\bone}{\text{\usefont{U}{bbold}{m}{n}1}}
\renewcommand{\leq}{\leqslant}
\renewcommand{\geq}{\geqslant}

\theoremstyle{plain}
\newtheorem{theorem}{Theorem}[section]
\newtheorem{corollary}[theorem]{Corollary}
\newtheorem{proposition}[theorem]{Proposition}
\newtheorem{lemma}[theorem]{Lemma}
\newtheorem{fact}[theorem]{Fact}
\newtheorem*{theorem*}{Theorem}
\newcounter{nstep}
\newtheorem{step}[nstep]{Step}
\AddToHook{env/proof/begin}{\setcounter{nstep}{0}}

\newcounter{maintheorem}

\newtheorem{mainth}[maintheorem]{Theorem}

\numberwithin{equation}{section}

\theoremstyle{definition}
\newtheorem{example}[theorem]{Example}
\newtheorem{defi}[theorem]{Definition}

\theoremstyle{remark}
\newtheorem{rem}[theorem]{Remark}

\makeatletter
\renewcommand{\tocsection}[3]{%
	\indentlabel{\@ifnotempty{#2}{\bfseries\ignorespaces#1 #2\quad}}\bfseries#3}
\renewcommand{\tocsubsection}[3]{%
	\indentlabel{\@ifnotempty{#2}{\ignorespaces#1 #2\quad}}#3}

\newcommand\@dotsep{4.5}
\def\@tocline#1#2#3#4#5#6#7{\relax
	\ifnum #1>\c@tocdepth 
	\else
	\par \addpenalty\@secpenalty\addvspace{#2}%
	\begingroup \hyphenpenalty\@M
	\@ifempty{#4}{%
		\@tempdima\csname r@tocindent\number#1\endcsname\relax
	}{%
		\@tempdima#4\relax
	}%
	\parindent\z@ \leftskip#3\relax \advance\leftskip\@tempdima\relax
	\rightskip\@pnumwidth plus1em \parfillskip-\@pnumwidth
	#5\leavevmode\hskip-\@tempdima{#6}\nobreak
	\leaders\hbox{$\m@th\mkern \@dotsep mu\hbox{.}\mkern \@dotsep mu$}\hfill
	\nobreak
	\hbox to\@pnumwidth{\@tocpagenum{\ifnum#1=1\bfseries\fi#7}}\par
	\nobreak
	\endgroup
	\fi}
\AtBeginDocument{%
	\expandafter\renewcommand\csname r@tocindent0\endcsname{0pt}
}
\def\l@subsection{\@tocline{2}{0pt}{2.5pc}{5pc}{}}
\makeatother

\makeatletter
\g@addto@macro\bfseries{\boldmath} 
\makeatother

\title[The structural theorem and locally finite tilings]{Polyhedral normed spaces: the structural theorem and locally finite tilings}

\author[C.A.~De~Bernardi]{Carlo Alberto De Bernardi}
\address[C.A.~De~Bernardi]{Dipartimento di Matematica per le Scienze economiche, finanziarie ed attuariali, Universit\`a Cattolica del Sacro Cuore, 20123 Milano, Italy\newline
\href{https://orcid.org/0000-0002-9654-1324}{ORCID: \texttt{0000-0002-9654-1324}}}
\email{carloalberto.debernardi@unicatt.it, carloalberto.debernardi@gmail.com}

\author[H.~del~R\'{\i}o]{Helena del R\'{\i}o}
\address[H.~del~R\'{\i}o]{Department of Mathematical Analysis and Institute of Mathematics (IMAG), University of Granada, E-18071 Granada, Spain \newline
\href{https://orcid.org/0009-0004-5078-6993}{ORCID: \texttt{0009-0004-5078-6993}}}
\email{helenadelrio@ugr.es}

\author[T.~Russo]{Tommaso Russo}
\address[T.~Russo]{Universit\"{a}t Innsbruck, Department of Mathematics, Technikerstra\ss e 13, 6020 Innsbruck, Austria \newline
\href{https://orcid.org/0000-0003-3940-2771}{ORCID: \texttt{0000-0003-3940-2771}}}
\email{tommaso.russo@uibk.ac.at, tommaso.russo.math@protonmail.com}

\author[J.~Somaglia]{Jacopo Somaglia}
\address[J.~Somaglia]{Politecnico di Milano, Dipartimento di Matematica, Piazza Leonardo da Vinci 32, 20133 Milano, Italy \newline
\href{https://orcid.org/0000-0003-0320-3025}{ORCID: \texttt{0000-0003-0320-3025}}}
\email{jacopo.somaglia@polimi.it}


\begin{document}

\begin{abstract} A fundamental result due to Fonf (1981) asserts that the unit sphere of every polyhedral Banach space is covered by true faces of the unit ball. The principal aim of our paper is to study the validity of the same result for polyhedral normed spaces and present some applications. Our first main result is that if $X$ is a polyhedral normed space with property ($\Delta$), then its unit sphere is covered by algebraic true faces. It then follows that if the space is additionally (VI)-polyhedral, then its unit sphere is covered by genuine true faces. We also present several counterexamples showing that these results are optimal in a strong sense. For instance, we show that there exist (V)-polyhedral normed spaces whose unit ball doesn't have any algebraic true face at all, and that there exist polyhedral normed spaces with ($\Delta$) whose unit sphere is not covered by true faces. Further, we give an example of a polyhedral normed space such that the set of strongly exposed points of the unit ball is dense in the sphere, and some results and counterexamples concerning boundaries of polyhedral normed spaces. Our principal application of this material involves locally finite tilings of normed spaces, and we show that a normed space admits a locally finite tiling (by bounded convex bodies) if and only if it admits a (VI)-polyhedral norm with ($\Delta$). In particular, such a tiling exists in every polyhedral Banach space with ($\Delta$), which generalises a result of Fonf.
\end{abstract}
\maketitle
\tableofcontents

\section{Introduction}
A closed, bounded convex set $C$ in $\R^n$ is a \emph{polytope} if it has finitely many extreme points; equivalently, it is a finite intersection of closed half-spaces. Polytopes are ubiquitous objects that have found applications in several areas of mathematics, such as algebraic topology and algebraic geometry, combinatorics, and optimisation, just to name a few. We refer, \emph{e.g.}, to \cites{Coxeter, Grunbaum, HRGZ, Kaibel, Ziegler} for introductions to the study of polytopes and their applications. It is therefore not surprising that several attempts have been made to define polytopes also in infinite dimensions, \cites{AmirDeutsch, DP-Modena, DP-Rocky, InfPoly, Klee60}. The most successful definition is the one given by Klee in \cites{Klee60} (see also \cite{FLP_handbook}*{Section~6}): a closed, bounded convex set $C$ in a normed space $X$ is a \emph{polytope} if every finite-dimensional section of $C$ is a (finite-dimensional) polytope. A normed space $X$ is \emph{polyhedral} if its unit ball is a polytope. By a deep result of Klee \cite{Klee_2dim}*{Theorem~4.7}, $X$ is polyhedral if and only if the unit ball of every $2$-dimensional subspace of $X$ is a polygon. The simplest and most important example of a polyhedral Banach space is the space $c_0$. Conversely, it is clear that strictly convex or smooth normed spaces are not polyhedral, hence no $L_p(\mu)$, $1<p <\infty$, is polyhedral. Further, $\ell_1$ contains a $2$-dimensional smooth subspace \cites{KadetsFonf, Lind_ell1_smooth}, thus an $L_1(\mu)$ space is polyhedral if and only if it is finite dimensional. (The fact that $\ell_1$ is not polyhedral also follows from Lindenstrauss' result that no dual space is polyhedral, \cite{Lindenstrauss}.) Moreover, it is a folklore fact that $c$ is not polyhedral (see, \emph{e.g.}, \cite{GMZ_poly}, \cite{FPST1}*{Lemma~3}, or Example~\ref{ex: c not poly}), so a $C(K)$ space is polyhedral if and only if $K$ is finite.

Despite the seemingly innocent finite-dimensional definition, it turns out that polyhedral Banach spaces admit extremely strong structural properties (both of isometric and of isomorphic nature). To begin with, each polyhedral Banach space $X$ admits a boundary $\B$ of cardinality equal to $\textup{dens}(X)$, which also satisfies $B_{X^*}= \cconv(\B)$, \cite{fonfstruttnew}*{Theorems~1.4 and~3.9}. As a consequence, $\textup{dens}(X)= \textup{dens}(X^*)$ and, in particular, every polyhedral Banach space is an Asplund space. Moreover, every polyhedral Banach space is $c_0$-saturated, \cite{fonfstrutt}; as a consequence, no dual space is isomorphically polyhedral. Notice however that there are $c_0$-saturated (separable) Asplund spaces that are not isomorphically polyhedral, \cite{Leung}. In the separable case even stronger results are available. In fact, as a consequence of the existence of a countable boundary, separable polyhedral Banach spaces admit a $C^\infty$-smooth LFC norm \cite{Hajek_PAMS}, as well as an analytic norm \cite{devillefonfhajek}*{Corollary~3.2}. Furthermore, a separable Banach space admits a polyhedral norm if and only if it admits an LFC norm \cites{fonftiling, Hajek_PAMS}, if and only if it admits a locally finite tiling by bounded convex bodies \cite{fonftiling}*{Theorem~2}.

Let us also mention in passing that polyhedral Banach spaces have also been considered in connection to several topics in Banach space theory, such as extensions of compact operators \cites{Lazar, Lind_Memoires}, proximinality \cites{FonfLindVes, GodefroyIndu}, $L_1$-preduals \cites{CMPP, CMPV, CP_LindPoly}, extremal structure \cites{Schreier, DEPOLY, GMZ_poly}, Orlicz spaces \cites{HJ_Orlicz, Leung}, twisted sums \cites{CP_Hepheastus, CS_Twist1, CS_Twist2}, projection constants \cite{Basso}, diameter two properties \cite{GinesAbraham}. There also are several results in the literature aimed at constructing more examples or classes of polyhedral Banach spaces. For instance, Fonf \cite{Fonf_ordinals} proved that the Banach space $C([0,\alpha])$ admits a polyhedral norm, for each ordinal $\alpha$ (see also \cite{FPST1}*{Corollary~8}). Some more general methods to construct polyhedral norms in non-separable Banach spaces have been developed in \cites{AFST, FPST1, FPST2, Smith}. There even are examples of polyhedral Banach spaces that admit quotients isomorphic to $\ell_p$, $1<p <\infty$, \cites{Gasparis, GasparisII}.

As it turns out, essentially all the results mentioned in the previous paragraphs depend, explicitly or implicitly, upon a fundamental result due to Fonf \cites{fonfstrutt, fonfstruttnew} (see also \cites{FLP_handbook, veselystrutt} or \cite{HJ}*{Chapter~5, Theorem~103} for alternative proofs), that is nowadays known as the `structural theorem' and whose statement we recall here.

\begin{theorem*}[The structural theorem, Fonf \cites{fonfstrutt, fonfstruttnew}] If $X$ is a polyhedral Banach space, $S_X$ is covered by true faces of $B_X$. As a consequence, the set 
\[ \B_0\coloneqq \{f\in B_{X^*}\colon f^{-1}(1)\cap B_X \text{ is a true face}\} \]
is a boundary for $X$, which is contained in each boundary of $X$ and it coincides with $\wstr(B_{X^*})$. Further, $|\B_0|= \textup{dens}(X)$ and $B_{X^*}= \cconv(\B_0)$.
\end{theorem*}

Despite all the above-mentioned results concerning polyhedral Banach spaces, there is no obvious reason why one should include completeness in the definition of polyhedrality; for instance, all the notions of polyhedrality studied in \cites{DP-Modena, DP-Rocky} are defined for normed spaces. It is therefore natural to wonder to what extent the structural theory remains valid in absence of completeness, in the spirit, \emph{e.g.}, of \cites{dantashajekrusso, devillefonfhajek, Hajek_PAMS, HR_JFA}. The discussion in the previous paragraphs indicates that the first result to consider in this regard is undoubtedly the structural theorem. The second main motivation for us was Fonf's tiling result \cite{fonftiling}: in fact, Fonf's argument uses (IV)-polyhedrality in his proof, also known as property ($*$), and part of our motivation was also to understand the exact type of polyhedrality required, for possible extensions to non-separable spaces.

Interestingly, when considering the validity of the structural theorem for polyhedral normed spaces, two significant differences from the complete case appear. First, one has to distinguish between true faces (sets of the form $B_X\cap H$, where $H$ is a supporting hyperplane, that have non-empty relative interior in $H$) and algebraic true faces (where the topological interior is replaced by the algebraic one). Second, the result is no longer true for all polyhedral spaces and stronger notions of polyhedrality ought to be considered. (All the polyhedrality notions relevant for our paper, as well as undefined terminology mentioned in this section, will be defined in Section~\ref{sec: prelim}.) The result we obtain reads as follows.

\begin{mainth}\label{mth: structure thm} If $X$ is a \textup{(K)}-polyhedral normed space with property \textup{($\Delta$)}, then $S_X$ is covered by algebraic true faces of $B_X$. Further, if $X$ is additionally \textup{(VI)}-polyhedral, then $S_X$ is covered by true faces of $B_X$.
\end{mainth}

As a particular case, since (IV)-polyhedral normed spaces have property ($\Delta$), we obtain that the unit sphere of a (IV)-polyhedral normed space is covered by true faces. As a consequence of this version of the structural theorem, we obtain some results of isomorphic nature, in absence of completeness: for instance, if $X$ is a separable (VI)-polyhedral normed space with ($\Delta$), then $X$ has a $C^\infty$-smooth and LFC norm, as well as an analytic norm. Both results follow from the fact that such an $X$ admits a norm with countable boundary (Corollary~\ref{cor: C^infty renorm}). Further, in the same section we also prove that, for a polyhedral normed space $X$, $\wexp(B_{X^*})$ is a boundary if and only if $S_X$ is covered by algebraic true faces (Theorem~\ref{thm: bdry iff alg-ST}). Surprisingly, the corresponding characterisation with $\wstr(B_{X^*})$ and true faces does not hold (Remark~\ref{rmk: str exp vs true faces}).

Part of the interest of Theorem~\ref{mth: structure thm} in our opinion also lies in the fact that, unlike \cites{fonfstrutt, fonfstruttnew, veselystrutt}, our argument is entirely based on geometric considerations, without any analytic tool such as differentiation theorems.

Our main application of the theorem is postponed to the subsequent Section~\ref{sec: tiling}, where we investigate possible extensions of Fonf's tiling result \cite{fonftiling} to polyhedral normed spaces (particularly non-separable ones); this was actually the original motivation for the research in our paper.

\begin{mainth}\label{mth: tiling} A normed space $X$ admits a locally finite tiling (by bounded convex bodies) if and only if it admits a renorming that is \textup{(VI)}-polyhedral and has \textup{($\Delta$)}; equivalently, if and only if it admits a renorming that is \textup{(K)}-polyhedral and LFC.
\end{mainth}

In the course of the proof, we also observe that each body in a locally finite tiling is necessarily a polytope (Corollary~\ref{cor: locally finite gives Delta QP}), generalising an observation of \cite{KleeTri}, who proved the same for finite-dimensional spaces. Theorem~\ref{mth: tiling} provides us with a large collection of normed spaces that admit locally finite tilings\footnote{For simplicity, in the rest of the Introduction, we tacitly assume that all elements of all tilings are bounded convex bodies.}, for instance all normed spaces that are the linear span of a biorthogonal system (Corollary~\ref{cor: b.s. gives tiling}); intriguingly, it also yields an example of a non-separable normed space that admits a countable locally finite tiling (Example~\ref{ex: ell_Infty^F}), which cannot occur in a Banach space (Fact~\ref{fact: card of tiling in Banach}).

Even though the main focus of our paper is on normed spaces, our result even has strong consequences for Banach spaces. In fact, it implies that a Banach space admits a locally finite tiling if and only if it admits a polyhedral norm with ($\Delta$). Besides being a direct generalisation of Fonf's result \cite{fonftiling}, it clarifies which is the correct geometric assumption underlying the existence of locally finite tilings. Furthermore, it connects several apparently unrelated problems in renorming theory. In fact, it is unknown if every polyhedral Banach space admits a polyhedral norm with ($\Delta$) and it is also unknown if each polyhedral Banach space admits a polyhedral LFC norm. Our result shows that these two questions are one and the same, and that they are also equivalent to the existence of a locally finite tiling. Further, as LFC norms are most frequently employed as a tool to construct $C^\infty$-smooth ones, it is also conceivable that our result will be instrumental in advancing the understanding of smooth renormings of non-separable polyhedral Banach spaces. We refer to \cites{BibleSmith, dantashajekrusso, Smith, ST_JFA19} for some results about smooth and polyhedral renormings in the non-separable context.

Finally, Theorem~\ref{mth: structure thm} and Theorem~\ref{mth: tiling} seem to indicate the importance of (VI)-polyhedrality + ($\Delta$), especially in the renorming theory of non-separable polyhedral spaces. Notice that (VI)-polyhedrality + ($\Delta$) is a weakening of the more extensively studied (IV)-polyhedrality and, in a certain sense, a local variant of it.
\smallskip

In the second part of the paper (Section~\ref{sec: counterexamples}), we return to the structural theorem (Theorem~\ref{mth: structure thm}) with an opposite perspective and we construct several counterexamples that are collectively intended to show that our result in Theorem~\ref{mth: structure thm} is optimal in a strong sense. A summary of the results that we prove reads as follows.

\begin{mainth}\label{mth: examples} There exist polyhedral normed spaces whose unit sphere is not covered by true faces. More specifically:
\begin{enumerate}
    \item\label{mth: K Delta} There exists a \textup{(K)}-polyhedral normed space with property \textup{($\Delta$)} and whose unit sphere is not covered by the true faces (Section~\ref{sec: K+Delta no ST});
    \item\label{mth: V} There exists a \textup{(V)}-polyhedral normed space whose unit sphere doesn't have any algebraic true face (Section~\ref{sec: countable dim});
    \item\label{mth: str exp} There exists a \textup{(K)}-polyhedral normed space $X$ such that $\str(B_X)$ is dense in $S_X$ (Section~\ref{sec: str exp dense}).
\end{enumerate}
\end{mainth}

Therefore, combining the results of Theorem~\ref{mth: structure thm} and Theorem~\ref{mth: examples} we obtain the following situation. Suppose that $X$ is a polyhedral normed space with ($\Delta$). Then, already the weakest form of polyhedrality, namely (K)-polyhedrality, is sufficient to imply that $S_X$ is covered by algebraic true faces. Instead, (VI)-polyhedrality implies that the sphere is covered by genuine true faces and, if (VI)-polyhedrality is relaxed to the next weaker polyhedrality notion that we consider, namely (K)-polyhedrality, then the result fails to hold. Without the explicit assumption of ($\Delta$), (IV)-polyhedrality implies the validity of the structural theorem and relaxing (IV)-polyhedrality to (V)-polyhedrality doesn't even imply the algebraic version of the structural theorem. The following diagrams summarise the previous discussion.

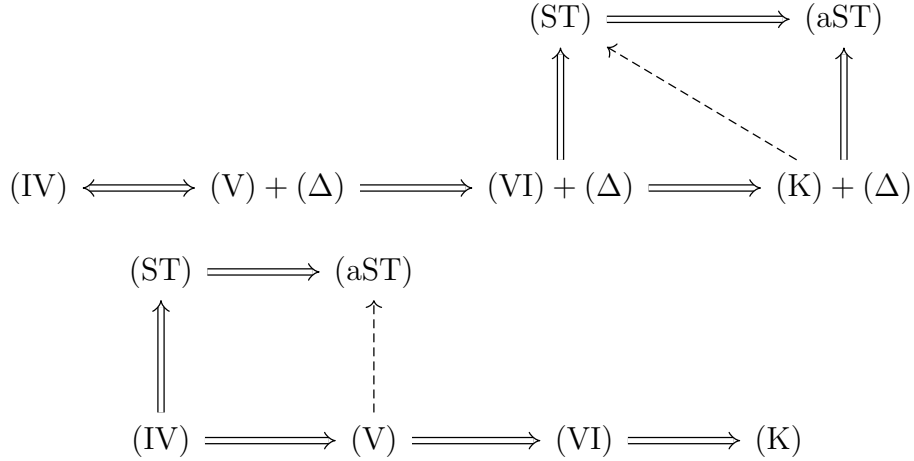
\begin{figure}[H]\begin{center}
\begin{tikzcd}[column sep=1.5cm, row sep=1.5cm]
& & (\mathrm{ST}) \arrow[r, Rightarrow] &  (\mathrm{aST}) \\
(\mathrm{IV}) \arrow[r, Leftrightarrow] & 
(\mathrm{V})+(\Delta) \arrow[r, Rightarrow] & 
(\mathrm{VI})+(\Delta) \arrow[u, Rightarrow] \arrow[r, Rightarrow] & 
(\mathrm{K})+(\Delta) \arrow[u, Rightarrow] \arrow[ul, rightarrow, dashed]
\end{tikzcd}
\bigskip

\begin{tikzcd}[column sep=1.5cm, row sep=1.5cm]
(\mathrm{ST}) \arrow[r, Rightarrow] &  (\mathrm{aST}) \\
(\mathrm{IV}) \arrow[r, Rightarrow] \arrow[u, Rightarrow] & 
(\mathrm{V})  \arrow[r, Rightarrow] \arrow[u, rightarrow, dashed] & 
(\mathrm{VI}) \arrow[r, Rightarrow] & 
(\mathrm{K}) 
\end{tikzcd}
\caption{The top diagram concerns polyhedral spaces with ($\Delta$), the bottom one general polyhedral spaces. Dashed arrows mean that the implication does not hold. (ST) means that the structural theorem holds, while (aST) means that the algebraic version holds.}
\end{center}\end{figure}

However, the results that we obtain in Section~\ref{sec: counterexamples} are actually much stronger than claimed in Theorem~\ref{mth: examples}. For instance, while a normed space as in \ref{mth: V} is easy to construct (it is enough to take the space $\ell_{1,0}$ of finitely supported sequences with the $\ell_1$ norm, Section~\ref{sec: ell1,0}), the result of Section~\ref{sec: countable dim} shows that every normed space of countable dimension has an equivalent (V)-polyhedral norm whose unit sphere doesn't have any true face. The same comment applies to the space in \ref{mth: str exp}, which can be taken to be a suitable renorming of any normed space of countable dimension. The normed space in \ref{mth: K Delta} is a renorming of $c_{00}$ and it also offers two more phenomena that can't happen in polyhedral Banach spaces. First, it is possible that in a polyhedral normed space a point is G\^ateaux smooth, but not Fr\'echet smooth (Example~\ref{ex: (K) + (Delta) not (VI)}); second, quite surprisingly, a point can be Fr\'echet smooth without being in the interior of any true face (Example~\ref{ex: Frechet vs true face}). In a forthcoming paper \cite{DDRS}, we will present additional examples, involving in particular non-separable spaces, and examples of (K)-polyhedral normed spaces with ($\Delta$) and whose unit sphere doesn't have any true face at all (improving \ref{mth: K Delta}).

The significance of \ref{mth: str exp} is that a normed space whose unit sphere is covered by true faces clearly can't satisfy the property in \ref{mth: str exp}, hence \ref{mth: str exp} gives another witness of the failure of the structural theorem. Moreover, it yields an example of a polyhedral normed space with `many' extreme points, in a strong sense. The construction in Section~\ref{sec: countable dim} also goes in this direction, as the space $X$ constructed there has the additional property that $\conv(\ext B_X)= B_X$ (notice that we only take the convex hull, and not its closure). These results thus complement several existing examples of polyhedral Banach spaces whose ball has many extreme points, as \cites{Schreier, DEPOLY, DePolyII}. In this respect, recall Lindenstrauss' problem \cite{Lindenstrauss} of whether the unit ball of a polyhedral Banach space can be the closed convex hull of its extreme points, which was solved in \cite{DEPOLY}. 
\smallskip

Besides the sections whose contents we already introduced, in Section~\ref{sec: prelim} we present the definitions of all undefined notions from the Introduction, and additional necessary ones; we also recall in detail the various notions of polyhedrality that we need and some basic facts about them. Finally, since several results from the literature that we require are only stated and proved for polyhedral Banach spaces, or with significantly different notation, and for possible future reference, in Appendix~\ref{sec: appendix} we give a more detailed outline of known results about polyhedrality that we use.

\section{Preliminaries and notation}\label{sec: prelim}
Throughout the paper, we only consider real normed spaces. For a normed space $X$, $B_X$ and $S_X$ denote its closed unit ball and unit sphere, respectively; the topological dual of $X$ is denoted $X^*$. In case an equivalent norm $\nn$ is introduced on $X$, the unit ball in the new norm is indicated by $B_{(X,\nn)}$ and similarly for the unit sphere $S_{(X,\nn)}$. Given points $x,y\in X$, we write $[x,y]$ for the closed segment joining $x$ and $y$, and $(x,y)$ for the relatively open one. $B(x,r)$ denotes the closed ball centred at $x$ of radius $r>0$. Accordingly, a \emph{neighbourhood} of a point $x$ is any set containing $x$ in its interior (namely, we don't assume that neighbourhoods are open sets). For a subset $C$ of $X$, we write $\R^+ C\coloneqq \{ax\colon a>0,\ x\in C\}$.

A \emph{biorthogonal system} in a normed space $X$ is a family $\{e_\alpha; e^*_\alpha\}_{\alpha\in \Gamma}\subseteq X\times X^*$ such that $e^*_\alpha (e_\beta)= \delta_{\alpha, \beta}$, for $\alpha, \beta\in \Gamma$. The \emph{support} of a vector $x$ relative to the biorthogonal system is $\supp(x)\coloneqq \{\alpha\in \Gamma\colon e^*_\alpha (x)\neq 0\}$. For a biorthogonal system indexed by $\N$, we just write $\{e_n; e^*_n\}_{n= 1}^\infty$. By \emph{dimension} of a normed space we understand the algebraic dimension, namely the cardinality of an algebraic basis. If $X$ has countable dimension, a standard result due to Markushevich \cite{Markushevich} (see, \emph{e.g.}, \cite{HMVZ}*{Lemma~1.21}) shows that $X= \spann\{e_n \}_{n= 1}^\infty$, for some biorthogonal system $\{e_n; e^*_n\}_{n= 1}^\infty$ in $X$.

We now recall some basic notions concerning convex sets and we refer to \cites{DGZ, FHHMZ, GMZ} for more details and notions we don't define here. A \emph{convex body} in a normed space is a closed convex set with non-empty interior (equivalently, the closure of a non-empty open convex set). Given a convex set $C$ in $X$, a point $x\in C$ is an \emph{extreme point} of $C$ if $x= \frac{y+z}{2}$ with $y,z\in C$ implies $x= y= z$ (namely, $x$ is not the mid-point of a non-trivial segment in $C$). We write $\ext C$ for the set of extreme points of $C$. The point $x$ is an \emph{exposed point} of $C$ if there is a functional $f\in X^*$ such that $f(y)< f(x)$ for all $y\in C$, $y\neq x$; in this case, the functional $f$ \emph{exposes} $x$. The point $x$ is \emph{strongly exposed} for $C$ if there is a functional $f$ that exposes it and $y_n\to x$ for each sequence $(y_n)_{n= 1}^\infty$ in $C$ such that $f(y_n)\to f(x)$. We write $\exp C$ and $\str C$ for the sets of exposed and strongly exposed points of $C$, respectively. We also require the $w^*$ versions of these concepts, which are obtained by requiring that the functional belongs to the predual, rather than the dual. More precisely, given a convex set $C$ in $X^*$, an element $f\in C$ is \emph{$w^*$-exposed} if there exists $x\in X$ such that $g(x)< f(x)$ for all $g\in C$, $g\neq f$. The functional $f$ is \emph{$w^*$-strongly exposed} if there exists $x\in X$ as above and such that $g_n\to f$ for each sequence $(g_n)_{n= 1}^\infty$ in $C$ such that $g_n(x)\to f(x)$. We use the notations $\wexp C$ and $\wstr C$ for these sets. Recall that, by the classical \v{S}mulyan Lemma (see, \emph{e.g.}, \cite{GMZ}*{Lemma 142}, or \cite{Megginson}*{Theorem 5.6.11}), the norm $\n$ is Fr\'echet differentiable at $x\in S_X$ if and only if there exists $f\in B_{X^*}$ that is $w^*$-strongly exposed in $B_{X^*}$ by $x$.

For the definitions of G\^{a}teaux and Fr\'echet smoothness we refer to the sources above. We only recall here the definition of LFC norm. A norm $\n$ on $X$ \emph{locally depends on finitely many coordinates} (is LFC, for short) if for each $x\in S_X$ there exist a neighbourhood $U$ of $x$ and functionals $f_1,\ldots,f_n\in X^*$ such that $\|y\|=\|z\|$ for every $y,z\in U$ with $f_k(y)=f_k(z)$ for all $k=1, \dots, n$.
\smallskip

For a subset $A$ of $X^*$, $A'$ stands for the set of all $w^*$-cluster points of $A$. Given $x\in S_X$, we write 
\[ \D_X(x)\coloneqq \{f\in S_{X^*}\colon f(x)=1\}, \]
that is, $\D_X(x)$ is the image of $x$ under the duality mapping (depending on the context, we also write $\D_{(X,\n)}(x)$, or just $\D(x)$). Note that, for each $x\in S_X$, $\D(x)$ is $w^*$-compact and convex; further, it is an extremal set. Standard arguments involving extremality and the Krein--Milman theorem give the following.

\begin{fact}\label{fact: basic D(x)} Let $X$ be a normed space and $x\in S_X$. Then:
\begin{enumerate}
    \item\label{i: ext D(x)=} $\D(x)\cap \ext B_{X^*}=\ext \D(x)$;
    \item\label{i: ext D(y)} if $y\in S_X$ is such that $\D(y)\subseteq \D(x)$, then $\ext \D(y)\subseteq \ext \D(x)$;
    \item\label{i: extend extreme} if $Y$ is a subspace of $X$, $x\in Y$, and $g\in \ext \D_Y(x)$, there exists $f\in \ext \D_X(x)$ such that $f\cut_Y=g$.
\end{enumerate}
\end{fact}

A set $\B \subseteq B_{X^*}$ is \emph{$1$-norming} for the normed space $X$ if, for all $x\in X$,
\[ \|x\|=\sup \{f(x)\colon f\in \B\}. \]
By the Hahn--Banach theorem, $\B$ is $1$-norming for $X$ if and only if $B_{X^*}= \cconv^{w^*}(\B)$. The
set $\B \subseteq B_{X^*}$ is a \emph{boundary} for $X$ if for every $x\in S_X$ there exists $f\in\B$ such that $f(x)=1$. The Krein--Milman theorem implies that $\ext B_{X^*}$ is always a boundary for $X$. Instead, the set $\wexp(B_{X^*})$ is not necessarily a boundary, but it is contained in all boundaries of $X$. If $X$ is a polyhedral Banach space, the structural theorem however implies that $\wexp(B_{X^*})$ is a boundary, which is then contained in every other boundary. This boundary is usually denoted by $\B_0$ and called \emph{the minimal boundary}. For a polyhedral normed space $X$, we will prove that $\wexp(B_{X^*})$ is a boundary if and only if $S_X$ is covered by algebraic true faces of $B_X$ (Theorem~\ref{thm: bdry iff alg-ST}). In our paper the notation $\B_0$ and the terminology \emph{the minimal boundary} will be used in reference to a polyhedral normed space only when its unit sphere is covered by algebraic true faces and for $\B_0= \wexp(B_{X^*})$.
\smallskip

We now recall the polyhedrality notions that are relevant for our paper. We refer to \cites{DP-Modena, DP-Rocky, FLP_handbook, InfPoly} for more definitions, references, and results. Since in some of these sources the results are only stated for polyhedral Banach spaces, we give a more detailed outline of the results we need in Appendix~\ref{sec: appendix}.

\begin{defi} Consider the following properties of a normed space $X$:
\begin{enumerate}
    \item[(IV)] $f(x)<1$ for every $x\in S_X$ and $f\in (\ext B_{X^*})'$;
    \item[(V)] $\sup \{f(x) \colon f \in \ext B_{X^*} \setminus \D(x)\} <1$ for every $x\in S_X$;
    \item[(VI)] every $x\in S_X$ has a neighbourhood $V$ such that, for every $y\in V\cap S_X$, the segment $[x,y]$ is contained in $S_X$;
    \item[(K)] the unit ball of every finite-dimensional subspace of $X$ is a polytope.
\end{enumerate}
For $\textup{j} \in \{\textup{IV, V, VI, K}\}$, $X$ is called (j)\emph{-polyhedral} if it satisfies property (j).
\end{defi}

\begin{rem}\label{rmk: poly def} We outline the following remarks about the previous definitions.
\begin{enumerate}
    \item The following implications were proved in \cite{DP-Rocky}:
    \[ \textup{(IV)} \implies \textup{(V)} \implies \textup{(VI)} \implies \textup{(K)}. \]
    It is also known that none of them can be reversed, \cites{DP-Rocky, InfPoly}.
    \item The notion of (K)-polyhedrality is the most general and was introduced by Klee in \cite{Klee60}. If a normed space $X$ is (K)-polyhedral it is often simply said that it is \emph{polyhedral} and we occasionally also follow this convention. Further, (IV)-polyhedrality is often called property ($*$).
    \item\label{item: equiv IV V VI} In the definition of (IV) or (V) one can replace the boundary $\ext B_{X^*}$ with any boundary for $X$ (Lemma~\ref{lemma: V poly for bdry}). Further, property (VI) admits the following reformulation: every $x\in S_X$ has a neighbourhood $V$ such that $\D(y)\subseteq \D(x)$ for every $y\in V\cap S_X$ (Lemma~\ref{lemma: VI iff D(y)}).
\end{enumerate}
\end{rem}

The following geometric condition will play a key role in our results. A normed space $X$ satisfies property ($\Delta$) if for every $x\in S_X$ the set $\ext \D(x)$ is finite. By Fact~\ref{fact: basic D(x)}\ref{i: ext D(x)=}, it is equivalent to ask that $\D(x)\cap \ext B_{X^*}$ is finite for each $x\in S_X$. Further, as in Remark~\ref{rmk: poly def}\ref{item: equiv IV V VI} above, one can also replace the boundary $\ext B_{X^*}$ with any boundary for $X$ (Fact~\ref{fact: Delta iff B cap D finite}).

\begin{fact}[\cite{InfPoly}*{Observation~3.5}]\label{fact: IV iff D + V} A normed space $X$ is \textup{(IV)}-polyhedral if and only if it is \textup{(V)}-polyhedral and satisfies \textup{($\Delta$)}.
\end{fact}

\begin{defi}\label{def: true faces} Given a convex body $C$ in a normed space $X$ and a supporting hyperplane $H$ for $C$, the face $F\coloneqq C\cap H$ is a \emph{true face} (\emph{algebraic true face}, respectively) of $C$ if $\inte_H(F)\neq \emptyset$ ($\ainte_H(F)\neq \emptyset$, respectively). Here and throughout the paper $\ainte_H(F)$ denotes the algebraic interior of $F$ relative to $H$. 
\end{defi}

\begin{rem}\label{rem: true faces} We list here a few basic observations that we require. Let $C$ be a closed convex set in a normed space $X$.
\begin{enumerate}
    \item\label{item: alg and true face coincide for banach} A simple `cone' argument shows that $\inte(C)= \ainte (C)$, whenever $\inte(C)\neq \emptyset$. Moreover, if $X$ is a Banach space, $\ainte(C)\neq \emptyset$ implies $\inte(C)\neq \emptyset$, by the Baire category theorem. Thus, in a Banach space there is no distinction between the concepts of true face and algebraic true face. This is not the case for incomplete normed spaces, as our examples in Section~\ref{sec: counterexamples} illustrate.
    \item\label{item: interiors coincide} The condition $\inte_H(F)\neq\emptyset$ in the definition of true face can be replaced by the condition $\inte_{\partial C}(F)\neq\emptyset$. In fact, it actually holds that $\inte_{\partial C}(F)=\inte_{H}(F)$, since clearly $F= (\partial C)\cap H$.
    \item\label{item: distinct faces} Finally, if $F$ and $F'$ are distinct algebraic true faces of a convex body $C$, then $\ainte_H(F)\cap F'= \emptyset$ (where $H$ is the hyperplane such that $F=C\cap H$). (When $C$ is the unit ball, this for example follows from \ref{item: alg true face}$\implies$\ref{item: singleton} in Proposition~\ref{prop: poly VS smooth}.) Note that this implies the same assertion for true faces.
\end{enumerate}
\end{rem}

Finally, ahead of the results of Section~\ref{sec: tiling}, we recall that a \emph{tiling} $\cal T$ of a normed space $X$ is a collection of convex bodies that cover $X$ and have mutually disjoint interiors. A point $x\in X$ is a \emph{regular point} for $\cal T$ if it admits a neighbourhood that intersects only finitely many elements of $\cal T$. The tiling $\cal T$ is \emph{locally finite} if each $x\in X$ is a regular point for $\cal T$. We refer, \emph{e.g.}, to \cites{DRShilbert, DEVEtiling, fonftiling, Klee2} for information on tilings of normed spaces.

\begin{rem}\label{rmk: loc finite tiling} Let us list here two basic properties of locally finite tilings, \cites{fonftiling, Klee2}. Let $\cal T$ be a locally finite tiling of a normed space $X$ by bounded convex bodies.
\begin{enumerate}
    \item\label{i: protective} For each $x\in X$ the set $\cal T(x)\coloneqq \{C\in \cal T\colon x\in C\}$ is finite and there is a neighbourhood $U$ of $x$ that is disjoint from each tile in $\cal T\setminus \cal T(x)$ (which of course implies that $U\subseteq \bigcup_{C\in \cal T(x)}C$). In fact, let $V$ be a neighbourhood of $x$ that intersects only finitely many elements of $\cal T$, say $C_1,\dots, C_n$. Then,
    \[ U\coloneqq V\setminus \bigcup\{ C_j\colon j=1, \dots, n,\ x\notin C_j \} \]
    is a neighbourhood of $x$ (as each $C_j$ is closed).
    \item\label{i: regular} As a consequence of the previous item, if $x$ belongs to the boundary of some $C\in \cal T$, there is at least another $C'\in \cal T$ with $x\in C'$.
\end{enumerate}
\end{rem}

\section{The structural theorem}\label{sec: structural thm}
This section is devoted to our first main result, and we study sufficient conditions for a polyhedral normed space to satisfy the conclusion of the structural theorem, or its algebraic version. The main result of the section is Theorem~\ref{thm: structural in normed space}, which proves Theorem~\ref{mth: structure thm}. We also identify the correct counterpart of the fact that in a polyhedral Banach space $\wexp(B_{X^*})$ is the minimal boundary, namely we prove that, for a polyhedral normed space $X$, $\wexp(B_{X^*})$ is a boundary if and only if $S_X$ is covered by algebraic true faces.

We begin with a result connecting faces with smoothness, which, in the Banach space setting, is essentially \cite{veselystrutt}*{Lemma 2}.

\begin{proposition}\label{prop: poly VS smooth} Let $X$ be a polyhedral normed space, $x\in S_X$, and $f\in \D(x)$; also let $H$ be the hyperplane $H=f^{-1}(1)$. Consider the following conditions:
\begin{enumerate}
    \item\label{item: singleton} $\D(x)=\{f\}$;
    \item\label{item: 2-dim section} if $Y$ is a $2$-dimensional subspace  of $X$ containing $x$ then $x\not\in \ext{B_Y}$;
    \item\label{item: alg true face}  $x\in \ainte_H(H\cap B_X)$, i.e., $x$ is in the relative algebraic interior of $H\cap B_X$ in $H$;
    \item\label{item: G-smooth} $x$ is a G\^ateaux smooth point;
    \item\label{item: frechet} $x$ is a Fr\'echet smooth point;
    \item\label{item: true face}  $x\in \inte_H(H\cap B_X)$, i.e., $x$ is in the relative interior of $H\cap B_X$ in $H$.
\end{enumerate}
Then, 
\[ \ref{item: singleton}\iff\ref{item: 2-dim section}\iff\ref{item: alg true face}\iff\ref{item: G-smooth}\impliedby\ref{item: frechet}\impliedby\ref{item: true face}. \]
Moreover, if $X$ is \textup{(VI)}-polyhedral or $S_X$ is covered by true faces, then all the conditions \ref{item: singleton}--\ref{item: true face} are equivalent.
\end{proposition}

In particular, all the above items are equivalent when $X$ is complete, which also follows from \cite{veselystrutt}*{Lemma 2}. Further, the last part of the result shows that, if there exists a point in $S_X$ that is G\^ateaux smooth but not Fr\'echet smooth, then $S_X$ is not covered by true faces. In Section~\ref{sec: K+Delta no ST} we will give examples showing that in general polyhedral normed spaces the implications \ref{item: G-smooth}$\implies$\ref{item: frechet}$\implies$\ref{item: true face} might not hold.

\begin{proof} The implications \ref{item: singleton}$\iff$\ref{item: G-smooth}$\impliedby$\ref{item: frechet} are standard.

\ref{item: singleton}$\implies$\ref{item: 2-dim section}. Suppose that \ref{item: singleton} holds and let $Y$ be any 2-dimensional subspace of $X$ containing $x.$ By the Hahn--Banach theorem, $\D_Y(x)=\{f\cut_Y\}$, which implies that $x$ is not an extreme point of the polygon $B_Y$.

\ref{item: 2-dim section}$\implies$\ref{item: alg true face}. Fix any line $L\subseteq H$ such that $x\in L$ and consider the $2$-dimensional subspace $Y=\spann \{x, L\}$. Then, $L$ is a supporting hyperplane of $B_Y$ at $x$. As $x\notin \ext B_Y$, $L\cap B_Y$ is a non-trivial segment which contains $x$ as an interior point.

\ref{item: alg true face}$\implies$\ref{item: singleton}. Let $g\in \D(x)$, and take any $v\in \ker f$, $v\neq 0$. \ref{item: alg true face} implies that there exists $r>0$ such that $x\pm rv\in S_X$, whence $g(x\pm rv)\leq 1$. However, as 
\[ 1=g(x)=\frac{1}{2}g(x+rv)+\frac{1}{2}g(x-rv), \]
the only possibility is that $g(x\pm rv)=1$, thus $v\in \ker g$. Hence, $\ker f\subseteq \ker g,$ and so $g\in \spann \{f\}$. But, as $g(x)=f(x)=1$, it must be that $g=f$.

\ref{item: true face}$\implies$\ref{item: frechet}. Let $V\subseteq \inte_H(H\cap B_X)$ be a neighbourhood of $x$ in $H$. In particular, every $y\in V$ satisfies $\|y\|= f(y)= 1$. This implies that there exists a neighbourhood of $x$ in $X$, namely $\R^+V$, such that 
\[ \|z\|=f(z) \]
for every $z\in \R^+ V$. Let $r>0$ be such that $B(x,r)\subseteq \R^+ V$. Then,
\[ \frac{\|x+h\|-\|x\|-f(h)}{\|h\|}=\frac{1+f(h)-1-f(h)}{\|h\|}=0 \]
for each $h\in X$ with $\|h\|< r$, and so $x$ is a Fr\'echet smooth point.

Finally, we show that \ref{item: singleton}$\implies$\ref{item: true face} holds under each of the two additional assumptions. Suppose first that $X$ is (VI)-polyhedral. Then, by Lemma~\ref{lemma: VI iff D(y)} $x$ has a neighbourhood $V$ such that $\D(y)\subseteq \D(x)=\{f\}$ for every $y\in V\cap S_X$. In particular, $f(y)=1$ for every $y\in V\cap S_X$, proving that $x\in \inte_{S_X}(H\cap B_X)$. By Remark~\ref{rem: true faces}\ref{item: interiors coincide}, \ref{item: true face} holds.

Next, suppose that $S_X$ is covered by true faces. By assumption, $\D(x)=\{f\}$, so the only face $x$ belongs to is the face $\Gamma_f$ induced by $f$; further, as we already proved that \ref{item: singleton}$\implies$\ref{item: alg true face}, $x\in \ainte_H(\Gamma_f)$. On the other hand, $S_X$ is covered by true faces, so $\Gamma_f$ must be a true face, thus $\inte_H(\Gamma_f) \neq \emptyset$. However, as recalled in Remark~\ref{rem: true faces}\ref{item: alg and true face coincide for banach}, for a convex set $C$ with non-empty interior, $\inte (C)= \ainte(C)$. Thus, $x\in \inte_H(\Gamma_f)$, which is \ref{item: true face}.
\end{proof}

We now pass to the main result of the section.

\begin{theorem}\label{thm: structural in normed space} Let $X$ be a normed space with property \textup{($\Delta$)}. If $X$ is \textup{(K)}-polyhedral, then $S_X$ is covered by algebraic true faces of $B_X$. Moreover, if $X$ is \textup{(VI)}-polyhedral, then $S_X$ is covered by true faces of $B_X$.
\end{theorem}

The second part of the result follows from the first one, but it could also be deduced from the implication \ref{i: VI-Delta}$\implies$\ref{i: finite true faces} in Lemma~\ref{lemma: Delta-QP pts iff VI+Delta}.

\begin{proof} The second part of the theorem is a consequence of the first one, as in (VI)-polyhedral spaces algebraic true faces are true faces, by the last part of Proposition~\ref{prop: poly VS smooth}. For the first part, let $X$ be (K)-polyhedral and have property ($\Delta$). Since $X$ has property ($\Delta$), it is sufficient to prove that for every $k\in \N$ the following condition holds:
\begin{equation}\label{eq: algebraic true face}\tag{$F_k$}
\begin{split}
{\text{if $x\in S_X$ and $k=|\ext \D_X(x)|$, there exists $f\in\D_X(x)$}}\\
{\text{such that $f$ defines an algebraic true face of $B_X$.}\quad}
\end{split}
\end{equation}

Let us proceed by induction. By Proposition~\ref{prop: poly VS smooth}, ($F_1$) is clearly valid. Let $n\in\N$, suppose that ($F_k$) holds for $k=1,\ldots,n$, and let us prove that  ($F_{n+1}$) holds. Fix $x\in S_X$ such that $n+1=|\ext \D_X(x)|$. In particular, $|\ext \D_X(x)|>1$, thus $\D_X(x)$ is not a singleton. Thus, Proposition~\ref{prop: poly VS smooth} implies that there exists a 2-dimensional subspace $Y$ of $X$ such that $x\in\ext B_Y$. Since $B_Y$ is a polygon and $x\in\ext B_Y$, there exist $g_1,g_2\in S_{Y^*}$ such that $\ext \D_Y(x)=\{g_1,g_2\}$. Moreover, there exist $v_1,v_2\in S_Y$ such that $g_i^{-1} (1)\cap B_Y= [x,v_i]$ ($i=1,2$). Clearly, $g_1(v_2)<1$. Now, by Fact~\ref{fact: basic D(x)}\ref{i: extend extreme}, there exists $f_1\in\ext\D_X(x)$ such that $f_1\cut_Y=g_1$. Let $x_2$ be any point in the interval $(x,v_2)$ and note that:
\begin{enumerate}
    \item $\D_X(x_2)\subseteq \D_X(x)$ (and hence $\ext\D_X(x_2)\subseteq \ext\D_X(x)$, by Fact~\ref{fact: basic D(x)}\ref{i: ext D(y)});
    \item $f_1\not\in \D_X(x_2)$.
\end{enumerate}
So, $|\ext \D_X(x_2)|\leq n$ and by our induction hypothesis there exists $f\in\D_X(x_2)$ such that $f$ defines an algebraic true face of $B_X$. Since $f\in \D_X(x_2)\subseteq \D_X(x)$, condition ($F_{n+1}$) holds and the proof is concluded. 
\end{proof}

In particular, combining Fact~\ref{fact: IV iff D + V} with Theorem~\ref{thm: structural in normed space} we get that the unit sphere of any (IV)-polyhedral normed space is covered by true faces.

\begin{corollary}\label{cor: str thm for IV-poly} If $X$ is a \textup{(IV)}-polyhedral normed space, $S_X$ is covered by true faces.
\end{corollary}

\begin{rem}\label{rmk: true faces from hat X} Another sufficient condition for $S_X$ being covered by true faces is that the completion $\hat X$ of $X$ is polyhedral. In fact, $S_{\hat X}$ is covered by true faces of $B_{\hat X}$ and their restrictions to $S_X$ cover it. Further, $S_X$ is dense in $S_{\hat X}$, so the interior of each true face of $S_{\hat X}$ intersects $S_X$; thus, the restrictions of the true faces of $S_{\hat X}$ are true faces of $S_X$.
\end{rem}

\begin{example} Let us give a simple example to which the above remark applies, but Theorem~\ref{thm: structural in normed space} doesn't. Consider the renorming $\nn$ of $c_0$ given in \cite{InfPoly}*{Example~4.5}. It is shown there that $(c_0,\nn)$ is (K)-polyhedral and Step 3 there actually shows that $(c_{00},\nn)$ is not (VI)-polyhedral. Yet, its unit sphere is covered by true faces, by Remark~\ref{rmk: true faces from hat X}. In particular, even for normed spaces with countable dimension, the validity of the structural theorem does not imply (VI)-polyhedrality.
\end{example}

We next give a first application to the isomorphic theory of separable polyhedral spaces. Recall that every separable polyhedral Banach space admits a (IV)-polyhedral renorming (see, \emph{e.g.}, \cite{InfPoly}*{Theorem 1.6}). While we don't know if this is true for separable polyhedral normed spaces, we can prove it for (VI)-polyhedral normed spaces with ($\Delta$).

\begin{corollary}\label{cor: C^infty renorm} Every separable \textup{(VI)}-polyhedral normed space $(X,\n)$ with \textup{($\Delta$)} admits a \textup{(IV)}-polyhedral renorming (which approximates $\n$). Moreover, the norm $\n$ can be approximated by $C^\infty$-smooth LFC norms, and by analytic norms.
\end{corollary}

\begin{proof} If $X$ is a separable (VI)-polyhedral normed space with ($\Delta$), Theorem~\ref{thm: structural in normed space} implies that $S_X$ is covered by true faces. By separability, these true faces form a countable set, hence the minimal boundary of $X$ is countable; in particular, $(X,\n)$ has a countable boundary. Fact~\ref{fact: ctble boundary and IV} then yields that $X$ has a (IV)-polyhedral norm and the approximation claim.

For the second part, the fact that every norm with countable boundary can be approximated by a $C^\infty$-smooth LFC norm is proved in \cite{Hajek_PAMS}*{Theorem~1} (see, \emph{e.g.}, \cite{HJ}*{Theorem~5.99}); instead, the analytic approximation follows from the proof of \cite{devillefonfhajek}*{Theorem~3.1}. \end{proof}

In conclusion to this section, we move to the connection between algebraic true faces and $\wexp(B_{X^*})$ being a boundary. The characterisation we obtain in Theorem~\ref{thm: bdry iff alg-ST} in particular gives another sufficient condition for the validity of the algebraic form of the structural theorem, which we will use in Example~\ref{ex: no ctble bdry} for a space which doesn't have property ($\Delta$) (and thus Theorem~\ref{thm: structural in normed space} doesn't apply). We begin with a folklore fact valid for all normed spaces.

\begin{fact}\label{fact: bdry} Let $X$ be a normed space and $\B$ be a boundary for $X$. Then, the following are equivalent:
\begin{enumerate}
    \item\label{bdry: minimal} no proper subset of $\B$ is a boundary;
    \item\label{bdry: singleton} for each $f\in \B$ there exists $x\in S_X$ with $\D(x)= \{f\}$;
    \item\label{bdry: B= w-exp} $\B= \wexp(B_{X^*})$;
    \item\label{bdry: smallest} $\B$ is contained in each boundary of $X$.
\end{enumerate}
\end{fact}

\begin{proof} For \ref{bdry: minimal}$\implies$\ref{bdry: singleton}, fix any $f\in \B$. If there is no $x\in S_X$ with $\D(x)= \{f\}$, then for all $x\in S_X$, we have that $\D(x)\setminus \{f\}\neq \emptyset$. Thus, $\B\setminus \{f\}$ is also a boundary, against the minimality of $\B$. Next, \ref{bdry: singleton} just means that $\B\subseteq \wexp(B_{X^*})$, thus equality holds, as $\wexp(B_{X^*})$ is contained in each boundary, thus proving \ref{bdry: B= w-exp}. Finally, \ref{bdry: B= w-exp}$\implies$\ref{bdry: smallest} is again the fact that $\wexp(B_{X^*})$ is contained in each boundary, and \ref{bdry: smallest}$\implies$\ref{bdry: minimal} is obvious.
\end{proof}

\begin{theorem}\label{thm: bdry iff alg-ST} For a polyhedral normed space $X$, the following are equivalent:
\begin{enumerate}
    \item\label{i: alg ST holds} $S_X$ is covered by algebraic true faces of $B_X$;
    \item\label{i: w-exp bdry} $\wexp(B_{X^*})$ is a boundary for $X$;
    \item\label{i: bdry minimal} there exists a boundary $\B$ for $X$ that is minimal with respect to inclusion.
\end{enumerate}
Moreover, in case \ref{i: bdry minimal} holds, the boundary $\B$ coincides with $\wexp(B_{X^*})$.
\end{theorem}

\begin{proof} For \ref{i: alg ST holds}$\implies$\ref{i: w-exp bdry}, suppose that $S_X$ is covered by algebraic true faces, and let $\B$ be the set of functionals inducing these faces. Plainly, $\B$ is a boundary. Moreover, since each face has non-empty algebraic interior, by \ref{item: alg true face}$\implies$\ref{item: singleton} in Proposition~\ref{prop: poly VS smooth} for each $f\in \B$ there exists $x\in S_X$ with $\D(x)= \{f\}$. Thus, each $f\in \B$ is $w^*$-exposed, and so $\B\subseteq \wexp(B_{X^*})$. As $\B$ is a boundary, $\wexp(B_{X^*})$ is as well.

Next, \ref{i: w-exp bdry}$\implies$\ref{i: bdry minimal} holds because $\wexp(B_{X^*})$ is always contained in each boundary, so one can take $\B= \wexp(B_{X^*})$.

Finally, for \ref{i: bdry minimal}$\implies$\ref{i: alg ST holds}, let $\B$ be a boundary as in \ref{i: bdry minimal}. Fact~\ref{fact: bdry} implies that for each $f\in \B$ there is $x\in S_X$ with $\D(x)= \{f\}$; in turn, Proposition~\ref{prop: poly VS smooth} yields that $x$ belongs to the algebraic interior of the face induced by $f$. Thus, all faces induced by functionals in $\B$ are algebraic true faces and, by definition of boundary, they cover $S_X$. The fact that then $\B= \wexp(B_{X^*})$ is just Fact~\ref{fact: bdry}.
\end{proof}

\begin{rem}\label{rmk: str exp vs true faces} In parallel to the above equivalence it is very natural to hope for a boundary characterisation for the fact that $S_X$ is covered by true faces. The above result and the analogy with the case of Banach spaces give an obvious candidate, namely that $S_X$ is covered by true faces if and only if $\wstr(B_{X^*})$ is a boundary for $X$. However, this equivalence does not hold. More precisely, it remains valid that $\wstr(B_{X^*})$ is a boundary when $S_X$ is covered by true faces, but the converse is not true. The first fact follows as in the previous proof using \ref{item: true face}$\implies$\ref{item: frechet} of Proposition~\ref{prop: poly VS smooth}, while the second one will be shown in Example~\ref{ex: Frechet vs true face}.
\end{rem}

For a general normed space, it is clear that Theorem~\ref{thm: bdry iff alg-ST} does not hold; for instance, take $(\R^2,\n_2)$. More interesting is the fact that the theorem can also be false for spaces with a countable boundary.

\begin{example}\label{ex: c not poly} Consider the Banach space $c= C([0,\omega])$, and the boundary $\B\coloneqq\{\pm \delta_p \colon p\in [0,\omega]\}$. It is easy to check that $\B= \wexp(B_{C([0,\omega])^*})$, hence \ref{i: w-exp bdry} and \ref{i: bdry minimal} of Theorem~\ref{thm: bdry iff alg-ST} hold. On the other hand, the unit sphere of $C([0,\omega])$ is not covered by true faces (equivalently, by algebraic true faces, because of Remark~\ref{rem: true faces}\ref{item: alg and true face coincide for banach}). To see this, note that the set
\[ \bigcup_{p<\omega} \left(B_{C([0,\omega])}\cap (\pm \delta_p)^{-1}(1)\right)= \{f\in S_{C([0,\omega])}\colon |f(p)|=1 \text{ for some }p\in[0,\omega) \} \]
is a proper dense subset of $S_{C([0,\omega])}$. Thus, the face induced by $\delta_\omega$ is not a true face and \ref{i: alg ST holds} in Theorem~\ref{thm: bdry iff alg-ST} does not hold. Incidentally, this gives another proof that $c$ is not polyhedral.
\end{example}

\section{Locally finite tilings}\label{sec: tiling}
In this section we present our main application of the structural theorem from the previous section to the study of locally finite tilings; in particular, we prove Theorem~\ref{mth: tiling}. In the first part of the section, we study a property of a point in the boundary of a convex body that, when it holds for all points of the sphere, is equivalent to the fact that the space is (VI)-polyhedral and has ($\Delta$). This condition has been considered in different papers with different names, and part of our contribution also consists in systematising the notions.

We begin by recalling said geometric property, with the terminology taken from \cite{DRShilbert}*{Definition 3.7(ii)}, which in turn was inspired by \cite{dantashajekrusso}*{Lemma 2.6}. In the case of the unit ball, it was introduced as property (PH) by \cite{Reif}*{p.~143}, see also \cite{DP-Modena}*{Section~6}, or \cite{InfPoly}*{Section~5}. More precisely, in \cite{Reif} a space $X$ is said to be a (PH) space if each point of its unit sphere is a $\Delta$-QP point in our sense.

\begin{defi} Given a convex body $C$ in a normed space $X$ and a point $x\in \partial C$, we say that $x$ is a \emph{$\Delta$-QP point} for $C$ if there exist a neighbourhood $U$ of $x$ and functionals $f_1,\dots,f_n\in X^*$ such that 
\begin{equation} \label{eq: Delta-QP}
    C\cap U=\bigcap_{k=1}^n \{y\in U\colon f_k(y)\leq f_k(x)\}.
\end{equation}
\end{defi}

\begin{rem}\label{rmk: DeltaQP with 1} It is easy to check that, if $x$ and $f_1,\dots, f_n$ are as in the above definition, then each $f_k$ is a supporting functional for $C$ at $x$ (see the first paragraph in the proof of Lemma~\ref{lemma: DeltaQP iff locally finite faces}). Thus, if $0\in\inte(C)$, $f_k(x)\neq 0$ and, up to a scaling, one can assume that $f_k(x)=1$ for all $k=1,\dots, n$, which is exactly the formulation in \cite{dantashajekrusso}*{Lemma 2.6}.
\end{rem}

We start by a characterisation essentially stating that a point is a $\Delta$-QP point for $C$, if and only if around $x$ the boundary of $C$ is a finite union of true faces.

\begin{lemma}\label{lemma: DeltaQP iff locally finite faces} Let $C$ be a convex body in a normed space $X$ and $x\in \partial C$. Then, $x$ is a $\Delta$-QP point for $C$ if and only if there are an open neighbourhood $U$ of $x$ and finitely many true faces $\Gamma_1, \dots, \Gamma_n$ of $C$ that contain $x$ and such that
\begin{equation}\label{eq: bdry in n true faces}
    \partial C\cap U\subseteq \Gamma_1\cup \dots\cup \Gamma_n.
\end{equation}
Moreover, in this case, the only true faces of $C$ that intersect $U$ are $\Gamma_1, \dots, \Gamma_n$.
\end{lemma}

\begin{proof} We first prove the `$\implies$' implication, so we assume that $x\in \partial C$ is a $\Delta$-QP point. Take a neighbourhood $U$ and functionals $f_1,\dots, f_n\in X^*$ such that \eqref{eq: Delta-QP} holds; without loss of generality, we assume that $U$ is open. Consider the closed half-spaces $\Sigma_k\coloneqq \{y\in X\colon f_k(y)\leq f_k(x)\}$. We first note that $C\subseteq \Sigma_k$ for each $k=1,\dots, n$. In fact, otherwise there is $z\in C$ such that $f_k(z)> f_k(x)$. For each $\lambda \in (0,1)$, the vector $\lambda z+ (1-\lambda)x\in C $ also satisfies $f_k(\lambda z+ (1-\lambda)x) > f_k(x)$. Moreover, if $\lambda>0$ is chosen small enough, $\lambda z+ (1-\lambda)x\in U$, contradicting \eqref{eq: Delta-QP}.

Thus, $C\subseteq \Sigma_k$ for each $k$, hence the hyperplane $H_k\coloneqq \partial \Sigma_k$ is a supporting hyperplane for $C$ at $x$. Our assumption gives
\begin{equation}\label{eq: C inside Sigma_k}
    C\cap U= \bigcap_{k=1}^n (\Sigma_k\cap U),
\end{equation}
thus also
\[ \partial C\cap U\subseteq \bigcup_{k=1}^n (C \cap H_k) \]
(in fact, if $z\in \partial C\cap U \subseteq C\cap U$ belongs to no $H_k$, by \eqref{eq: C inside Sigma_k} it must belong to $\inte \Sigma_k$ for each $k$, thus again by \eqref{eq: C inside Sigma_k} also to $\inte C$, a contradiction). Hence, $\partial C\cap U$ is covered by finitely many closed sets, and so also by those ones having non-empty interior (in fact, the union of those sets with non-empty interior is by definition dense and closed as a finite union of closed sets). However, if $C \cap H_k$ has non-empty interior in $\partial C\cap U$, then it is a true face for $C$ by Remark~\ref{rem: true faces}\ref{item: interiors coincide}, which proves \eqref{eq: bdry in n true faces}.

Moreover, to show that no other true face intersects $U$, suppose that a true face $\Gamma$ intersects $U$, and let $H$ be the hyperplane inducing $\Gamma$. As $U$ is open, it must intersect the relative interior $\inte_H \Gamma$ of $\Gamma$. However, if $y\in \inte_H \Gamma \cap U$, by \eqref{eq: bdry in n true faces} there exists $k=1,\dots, n$ such that $y\in \Gamma_k$. Thus the relative interior of $\Gamma$ intersects $\Gamma_k$, and Remark~\ref{rem: true faces}\ref{item: distinct faces} gives $\Gamma= \Gamma_k$, as desired.

For the converse implication, we can assume that the neighbourhood $U$ is convex. Let $f_k\in X^*$ be the functional inducing the true face $\Gamma_k$; $f_k$ is a supporting functional for $C$ at the point $x$. Hence, we have
\[ C\cap U \subseteq \bigcap_{k=1}^n \{y\in U\colon f_k(y)\leq f_k(x)\}. \]
If the above containment is strict, we can find some point $y\in U$ such that $f_k(y)\leq f_k(x)$ for $k=1, \dots, n$, but $y\notin C$. Take any $z\in \inte (C) \cap U$ and let $w\in [y,z]$ be the point that belongs to $\partial C$. Since $z\in \inte C$, we have that $f_k(z)< f_k(x)$ for each $k=1,\dots, n$, hence $f_k(w)< f_k(x)$ as well (note that $w$ belongs to the interior of the segment $[y,z]$). On the other hand, the convexity of $U$ implies that $w\in U$. Thus, $w\in \partial C \cap U$, but it does not belong to any true face $\Gamma_k$, which contradicts \eqref{eq: bdry in n true faces}.
\end{proof}

As a consequence we obtain that, in a locally finite tiling, the points in the boundary of some tile are automatically $\Delta$-QP points. In turn, this yields that each body in a locally finite tiling is a polytope. This latter fact generalises \cite{KleeTri}*{Theorem~5.1}, where the same is proved for finite-dimensional spaces.

\begin{corollary}\label{cor: locally finite gives Delta QP} Let $\cal T$ be a locally finite tiling of a normed space $X$ by bounded convex bodies, and $C\in \cal T$. Then, each point $x\in \partial C$ is a $\Delta$-QP point for $C$. As a consequence, each body in $\cal T$ is a polytope.
\end{corollary}

\begin{proof} It is proved in \cite{DRShilbert}*{Fact~3.9} that if $x$ is a regular point for a tiling $\cal T$ and $x\in \partial C$ for some $C\in \cal T$, then $x$ is a $\Delta$-QP point for $C$. Since in a locally finite tiling each point is, by definition, regular, the first clause follows immediately.

For the second part, we will show that, if every point in $\partial C$ is a $\Delta$-QP point, then each finite-dimensional section of $C$ is a polytope\footnote{In the case when $C$ is the unit ball, this fact was proved in \cite{AmirDeutsch}*{Theorem~2.19}.} (namely, $C$ is a polytope itself). Since being a $\Delta$-QP point passes to subspaces, we can assume that $C$ is a bounded convex body in a finite-dimensional space. Lemma~\ref{lemma: DeltaQP iff locally finite faces} implies that, around each point, the boundary of $C$ is given by finitely many true faces. By compactness, the boundary of $C$ is a finite union of true faces, hence $C$ is a polytope.
\end{proof}

We now specialise the above results to the unit ball and obtain a characterisation of the normed spaces $X$ such that all points of $S_X$ are $\Delta$-QP points. Part of the characterisation is that those are exactly the (VI)-polyhedral normed spaces with property ($\Delta$), which was already known (see the references in the proof below). Let us mention that this characterisation is the motivation behind the name $\Delta$-QP point. In fact, in \cite{AmirDeutsch} (see also \cite{DP-Rocky}*{Section 5} or \cite{InfPoly}*{Section 2.5}) the notion of \emph{QP} point was introduced as a pointwise version of (VI)-polyhedrality, in the sense that $X$ is (VI)-polyhedral if and only if each point in $S_X$ is a QP point. Likewise, the notion of $\Delta$-QP point is the pointwise version of (VI)-polyhedrality with ($\Delta$).

\begin{lemma}\label{lemma: Delta-QP pts iff VI+Delta} For a normed space $X$, the following conditions are equivalent:
\begin{enumerate}
    \item\label{i: Delta-QP} every point of $S_X$ is a $\Delta$-QP point for $B_X$;
    \item\label{i: finite true faces} for each $x\in S_X$ there are a neighbourhood $U$ of $x$ and true faces $\Gamma_1, \dots, \Gamma_n$ containing $x$ and that cover $S_X\cap U$; further, no true face of $S_X$ different from $\Gamma_1, \dots, \Gamma_n$ intersects $U$;
    \item\label{i: VI-Delta} $X$ is a \textup{(VI)}-polyhedral normed space with property \textup{($\Delta$)};
    \item\label{i: K-LFC} $X$ is a \textup{(K)}-polyhedral normed space and its norm is LFC.
\end{enumerate}
\end{lemma}

\begin{proof} The equivalence between \ref{i: Delta-QP} and \ref{i: finite true faces} is just a particular case of Lemma~\ref{lemma: DeltaQP iff locally finite faces}, while \ref{i: Delta-QP}$\implies$\ref{i: K-LFC} is the content of \cite{dantashajekrusso}*{Lemma 2.6} (\emph{cf}. Remark~\ref{rmk: DeltaQP with 1}). Further, the equivalence between \ref{i: Delta-QP} and \ref{i: VI-Delta} is sketched in \cite{InfPoly}*{Proposition 5.2}: in fact, as we mentioned above, property (PH) is just condition \ref{i: Delta-QP}. The same argument is also used in \cite{FonfLindVes}*{Observation 1.7}. Since both sources formally give the proof for Banach spaces, we repeat the argument in the Appendix (Fact~\ref{fact: DeltaQP iff}).

Thus, we only need to prove that \ref{i: K-LFC} implies \ref{i: Delta-QP}. Let $x_0\in S_X$ and suppose that there exist a neighbourhood $U$ of $x_0$ and functionals $f_1,\dots,f_n\in X^*$ such that $\|y\|= \|z\|$ for every $y,z\in U$ with $f_j(y)= f_j(z)$ for $j=1,\dots,n$. Let $Y\coloneqq \bigcap_{j=1}^n\ker f_j$; notice that $x_0\not\in Y$ (in fact, otherwise for $\lambda\neq 1$ sufficiently close to $1$, one would have $\lambda x_0\in U$ and $f_j(\lambda x_0)= f_j(x_0)$, thus $\|\lambda x_0\|= \|x_0\|$, a contradiction). Let $Z$ be a finite-dimensional subspace of $X$ such that $x_0\in Z$ and $X=Y\oplus Z$. Let us denote by $P_Z$ the canonical projection on $Z$ (with $\ker(P_Z) =Y$). By passing to a suitable neighbourhood of $x_0$ contained in $U$, if necessary, and taking into account that $Z$ is a finite-dimensional polyhedral space, we can suppose without loss of generality that:
\begin{enumerate}[(1)]
    \item\label{i: form of U} $U$ is of the form $x_0+\e B_{Z\oplus_\infty Y}$, for some $\e>0$;
    \item\label{i: x_0 deltaQP in Z} there exist functionals $z^*_1,\dots,z^*_m\in Z^*$ such that 
\begin{equation} \label{eq: Delta-QP in Z}
    B_Z\cap U=\bigcap_{k=1}^m \{w\in x_0+\e U_Z\colon z^*_k(w)\leq z^*_k(x_0)\}.
\end{equation}
\end{enumerate}
Consider the functional $g_k\coloneqq z_k^*\circ P_Y\in X^*$ ($k=1, \dots, m$). We now show that
\begin{equation} \label{eq: Delta-QP in X}
    B_X\cap U=\bigcap_{k=1}^m \{x\in  U\colon g_k(x)\leq g_k(x_0)\},
\end{equation}
which shows that $x_0$ is a $\Delta$-QP point and proves \ref{i: Delta-QP}.

Suppose that $x\in U$. By \ref{i: form of U}, we can write $x=x_0+ y+ z$, where $y\in Y$, $z\in Z$, and $\max\{ \|y\|, \|z\| \}< \e$. By the definition of $Y$ we have that $f_j(x)= f_j(x_0+z)$ for all $j=1, \dots, n$; thus $\|x\|= \|x_0+ z\|$. Hence, $x\in B_X$ if and only if $x_0+ z \in B_Z$, which by \eqref{eq: Delta-QP in Z} is equivalent to
\[ z^*_k(x_0+z)\leq z^*_k(x_0), \qquad k=1, \dots, m. \]
Since $P_Z(x_0)=x_0$ and $P_Z(x)=x_0+z$, this last condition is equivalent to 
\[ g_k(x)\leq g_k(x_0), \qquad k=1,\dots,m. \]
This shows that \eqref{eq: Delta-QP in X} holds and completes the proof.
\end{proof}

We can now move to the main result of the section.

\begin{theorem}\label{thm: tiling iff VI+Delta} For a normed space $X$, the following conditions are equivalent:
\begin{enumerate}
    \item\label{i: loc finite tiling} $X$ admits a locally finite tiling by bounded convex bodies;
    \item\label{i: VI + Delta norm} $X$ admits an equivalent norm that is \textup{(VI)}-polyhedral and has \textup{($\Delta$)};
    \item\label{i: K + LFC norm} $X$ admits an equivalent norm that is \textup{(K)}-polyhedral and LFC.
\end{enumerate}
In this case, $X$ admits a locally finite tiling by (bounded) polytopes, and with the property that one polytope is symmetric with respect to the origin. 

Finally, if $X$ is infinite dimensional and $\n$ is a \textup{(VI)}-polyhedral norm on $X$ with property \textup{($\Delta$)}, the cardinality of the tiling can be taken to be equal to the cardinality of the minimal boundary of $(X,\n)$.
\end{theorem}

\begin{proof} The equivalence between \ref{i: VI + Delta norm} and \ref{i: K + LFC norm} is just Lemma~\ref{lemma: Delta-QP pts iff VI+Delta}. For the rest of the proof, we will proceed as follows. We first use a symmetrisation argument and Corollary~\ref{cor: locally finite gives Delta QP} to prove the implication \ref{i: loc finite tiling}$\implies$\ref{i: VI + Delta norm} (Step~\ref{step: tiling -> VI+Delta}). Next, we modify Fonf's `spiderweb' tiling from \cite{fonftiling}*{Theorem 2} to prove \ref{i: VI + Delta norm}$\implies$\ref{i: loc finite tiling}, proving along the way the second part of the theorem (Step~\ref{step: spiderweb}).

\begin{step}\label{step: tiling -> VI+Delta} The implication \ref{i: loc finite tiling}$\implies$\ref{i: VI + Delta norm}.
\end{step}
Suppose that $\cal T$ is a locally finite tiling of $X$ by bounded convex bodies. Up to a translation, we can assume without loss of generality that there is $\tilde{C}\in \cal T$ such that $0\in \inte \tilde{C}$. Consider the family
\[ \cal F\coloneqq \{C_1\cap (-C_2)\colon C_1, C_2\in \cal T,\, \inte (C_1\cap (-C_2))\neq \emptyset\}. \]
By definition, each element of $\cal F$ is a bounded convex body. We shall now show that $\cal F$ is additionally a locally finite tiling.

To begin with, we note that $\cal F$ is non-overlapping. Take distinct sets $C_1\cap (-C_2)$ and $C_1'\cap (-C_2')$ in $\cal F$ and assume for instance that $C_1 \neq C_1'$. Then, $\inte C_1$ and $\inte C_1'$ are disjoint, as $\cal T$ is a tiling. Thus the sets
\[ \inte(C_1\cap (-C_2))= \inte C_1 \cap \inte (-C_2) \quad\text{and}\quad \inte(C_1'\cap (-C_2'))= \inte C_1' \cap \inte (-C_2') \]
are also disjoint. The case when $C_2 \neq C_2'$ is identical, so $\cal F$ is non-overlapping.

Next, we show that $\cal F$ is a locally finite covering. Fix any $x\in X$. As $\cal T$ is locally finite, by Remark~\ref{rmk: loc finite tiling}\ref{i: protective} there are $r>0$ and $C_1,\dots, C_n\in \cal T$ such that $x\in C_j$ ($j= 1,\dots, n$) and the only elements in $\cal T$ that intersect $B(x,r)$ are $\{C_1, \dots, C_n\}$. By applying the same argument to the point $-x$ (and possibly taking a smaller $r>0$), there also are $C_1',\dots, C_m'\in \cal T$ such that $-x\in C_k'$ ($k= 1,\dots, m$) and the only elements in $\cal T$ that intersect $B(-x,r)$ are $\{C_1', \dots, C_m'\}$. Thus the only elements of $\cal F$ that can intersect $B(x,r)$ have the form $C_j \cap (-C_k')$, where $j=1,\dots, n$ and $k= 1,\dots, m$. This shows that $\cal F$ is locally finite. Moreover, since $\cal T$ is a covering, it also holds that
\[ B(x,r)\subseteq \bigcup_{j=1}^n C_j \qquad\text{and}\qquad B(x,r)\subseteq \bigcup_{k=1}^m (-C_k'). \]
Therefore, we deduce that
\begin{equation}\label{eq: ball covered by finitely many}
    B(x,r)\subseteq \bigcup_{\substack{1\leq j\leq n\\ 1\leq k\leq m}} C_j \cap (-C_k');
\end{equation}
hence, at least one of the sets in the right-hand side union must have non-empty interior, namely belong to $\cal F$. By construction, this set contains $x$, so $\cal F$ is a covering. This concludes the proof that $\cal F$ is a locally finite tiling.

Finally, recall that there exists $\tilde{C}\in \cal T$ such that $0\in \inte \tilde{C}$. Hence the set $B\coloneqq \tilde{C} \cap (-\tilde{C})$ belongs to $\cal F$ and it is additionally a symmetric neighbourhood of $0$. Therefore, there is a norm $\nn$ on $X$ such that $B$ is the unit ball of $(X,\nn)$. Since $\cal F$ is a locally finite tiling, Corollary~\ref{cor: locally finite gives Delta QP} yields that each point of $\partial B= S_{(X,\nn)}$ is a $\Delta$-QP point for $B$. In turn, Lemma~\ref{lemma: Delta-QP pts iff VI+Delta} implies that $(X,\nn)$ is a (VI)-polyhedral normed space with ($\Delta$), thus proving this implication.

\begin{step}\label{step: spiderweb} The implication \ref{i: VI + Delta norm}$\implies$\ref{i: loc finite tiling} and further properties of the tiling.
\end{step}
Suppose that $(X,\n)$ is a (VI)-polyhedral normed space with ($\Delta)$. By the structural theorem (Theorem~\ref{thm: structural in normed space}), its unit sphere is covered by true faces, hence the set $\B_0$ of all functionals that induce the true faces is the minimal boundary of $X$. For $f\in \B_0$, we indicate the corresponding true face by $\Gamma_f= B_X\cap f^{-1}(1)$. We then set
\[ T_{n,f}\coloneqq \{ax\colon x\in \Gamma_f,\, n\leq a\leq n+1 \} \]
and define the desired tiling to be
\[ \cal T\coloneqq \{T_{n,f}\colon n\in \N, f\in \B_0\}\cup\{B_X\}. \]
We shall show that $\cal T$ is a locally finite tiling by bounded convex bodies, which completes the proof. In fact, not only this yields the validity of \ref{i: loc finite tiling}, but each tile is also a polytope by Corollary~\ref{cor: locally finite gives Delta QP} and $B_X\in \cal T$ is obviously symmetric with respect to $0$. Finally, if $X$ is infinite dimensional, $\B_0$ is necessarily infinite. This implies that $|\cal T|= |\B_0|$, proving the last clause of the theorem.

To begin with, we can write
\[ T_{n,f}= f^{-1}([n,n+1]) \cap (\R^+ \Gamma_f), \]
where the cone $\R^+ \Gamma_f$ generated by $\Gamma_f$ is convex, because $\Gamma_f$ is convex. This shows that $T_{n,f}$ is closed and convex and that
\begin{equation}\label{eq: int T_nf}
\begin{split}
    \inte T_{n,f} &= \{ ax\colon x\in \inte \Gamma_f,\, n< a< n+1 \}\\
    &= \left\{ x\in X\colon \textstyle{\frac{x}{\|x\|}}\in \inte \Gamma_f,\, n< \|x\|< n+1 \right\}
\end{split}
\end{equation}
Thus, each element in $\cal T$ is a convex body. Next, $\cal T$ is a covering, because, for each $x\in X$ with $\|x\|>1$, there are $n\in \N$ with $n\leq \|x\|< n+1$ and $f\in \B_0$ such that $\frac{x}{\|x\|}\in \Gamma_f$ (as the true faces cover $S_X$). Thus, $x\in T_{n,f}$. Since, obviously, elements of norm at most $1$ belong to $B_X$, this shows that $\cal T$ is a covering. Further, the second equality in \eqref{eq: int T_nf} shows that $\cal T$ is non-overlapping. In fact, note first that no set $\inte T_{n,f}$ intersects $\inte B_X$. Moreover, if $x\in \inte T_{n,f}$, then $n$ is uniquely determined by the condition $n< \|x\|< n+1$, while $f$ is uniquely determined by $\frac{x}{\|x\|}\in \inte \Gamma_f$, as the sets $\inte \Gamma_f$ are mutually disjoint. Thus, $\cal T$ is a tiling by bounded convex bodies.

Finally, we show that $\cal T$ is locally finite. Fix any $x\in X$. If $\|x\|< 1$, it is clear that there is a  neighbourhood that only intersects one element $B_X$ of $\cal T$. Hence, we assume that $\|x\|\geq 1$ and we let $x'\coloneqq \frac{x}{\|x\|}$. As $X$ is (VI)-polyhedral and has ($\Delta$), by Lemma~\ref{lemma: DeltaQP iff locally finite faces} there are a neighbourhood $U$ of $x'$ and finitely many true faces $\Gamma_{f_1}, \dots, \Gamma_{f_k}$, where $f_1, \dots, f_k\in \B_0$, containing $x$ and with the following properties: the only true faces of $S_X$ that intersect $U$ are $\Gamma_{f_1}, \dots, \Gamma_{f_k}$, and 
\[ S_X\cap U\subseteq \Gamma_{f_1}\cup \dots\cup \Gamma_{f_k}. \]
Let $j\coloneqq \lfloor \|x\| \rfloor$ and consider the set
\[ W\coloneqq \left\{ ay\in X\colon y\in S_X\cap U, j-1< a< j+1 \right\}. \]
Since $x= \|x\|x'$, where $x'\in S_X\cap U$ and $j-1< \|x\|< j+1$, we see that $W$ is a neighbourhood of $x$. We shall show that $W$ intersects only finitely many elements of $\cal T$, which implies that $\cal T$ is locally finite.

Suppose that a tile $T_{n,f}\in \cal T$ intersects $W$ and take a point $ay\in T_{n,f}\cap W$. Since $ay\in T_{n,f}$, we have that $n\leq a\leq n+1$ and $y\in \Gamma_f$; likewise, $ay\in W$ implies that $j-1< a< j+1$ and $y\in S_X\cap U$. As only the true faces $\Gamma_{f_1}, \dots, \Gamma_{f_k}$ intersect $U$, we deduce that $f\in \{f_1, \dots, f_k\}$. Moreover, $n\in \{j-1,j\}$ (and $n=1$ if $j=1$). Thus, the only tiles in $\cal T$ that can intersect $W$ are
\[ \left\{ T_{n,f}\colon f\in \{f_1, \dots, f_k\},\ n\in \{j-1,j\} \right\}\cup \{B_X\}. \qedhere \]
\end{proof}

\begin{rem} In the proof in Step~\ref{step: tiling -> VI+Delta} that $\cal F$ is a covering we needed to use the fact that $\cal T$ is locally finite. In fact, if $\cal T$ is merely a tiling, it doesn't follow that $\cal F$ is a tiling as well. As a simple example in $\R$ consider
\[ \cal T= \{[n,n+1]\colon n\in \Z, n\neq 0\} \cup \left\{\left[ \frac{1}{n+1}, \frac{1}{n} \right]\colon n\in \N \right\}. \]
Then, 
\[ \cal F= \{ \pm [n,n+1]\colon n\in\N \} \cup \left\{ \pm\left[ \frac{1}{n+1}, \frac{1}{n} \right]\colon n\in \N \right\}, \]
which is not a covering, as $0\notin \bigcup_{I\in \cal F} I$. 
\end{rem}

We now move to some applications of Theorem~\ref{thm: tiling iff VI+Delta}. In all of them, the locally finite tiling we obtain is actually formed by polytopes, by Corollary~\ref{cor: locally finite gives Delta QP}. The first one pertains to polyhedral \emph{Banach} spaces, and it generalises Fonf's theorem from \cite{fonftiling}*{Theorem 2}.

\begin{corollary}\label{cor: Banach tiling iff} A Banach space admits a locally finite tiling by bounded convex bodies if and only if it admits a \textup{(K)}-polyhedral renorming with property \textup{($\Delta$)}.
\end{corollary}

This result is a particular case of Theorem~\ref{thm: tiling iff VI+Delta}, because, for a Banach space with ($\Delta$), (K)-polyhedrality is equivalent to (VI)-polyhedrality, \cite{InfPoly}*{Theorem 1.4}. However, it generalises Fonf's result for separable Banach spaces, since every separable polyhedral Banach space has a (IV)-polyhedral renorming. Moreover, it clarifies what is the exact polyhedrality required in order to obtain the result in the non-separable case.

\begin{corollary}\label{cor: ctble bdry -> tiling} Every \textup{(IV)}-polyhedral normed space admits a locally finite tiling by bounded convex bodies. Moreover, every normed space with a countable boundary admits a countable locally finite tiling by bounded convex bodies. In particular, this is the case for all normed spaces with countable dimension.
\end{corollary}

\begin{proof} The first clause is just a particular case of Theorem~\ref{thm: tiling iff VI+Delta}, as (IV)-polyhedral normed spaces have ($\Delta$), \emph{cf}. Fact~\ref{fact: IV iff D + V}. For the second part, if $X$ has a countable boundary, then it admits a (IV)-polyhedral equivalent norm that has a countable boundary (see Fact~\ref{fact: ctble boundary and IV} for details). Finally, if $X$ has countable dimension, it admits a (IV)-polyhedral norm with a countable boundary, by Lemma~\ref{lemma: ctble dim => IV}.
\end{proof}

The main drawback of Corollary~\ref{cor: ctble bdry -> tiling} is that it doesn't directly apply to all separable polyhedral normed spaces. In fact, while separable polyhedral Banach spaces have a countable boundary, there are separable polyhedral normed spaces that don't admit a countable boundary (Example~\ref{ex: no ctble bdry}). The last corollary we offer yields us a large amount of normed spaces that admit a locally finite tiling.

\begin{corollary}\label{cor: b.s. gives tiling} Let $\{e_\alpha; e^*_\alpha\}_{\alpha\in \Gamma}$ be a biorthogonal system in a normed space $X$. Then, the normed space $Y\coloneqq \spann \{e_\alpha\}_{\alpha\in \Gamma}$ admits a locally finite tiling by bounded convex bodies.
\end{corollary}

\begin{proof} It is proved in \cite{dantashajekrusso}*{Theorem A} that the linear span of every biorthogonal system admits a (K)-polyhedral and LFC norm. (Incidentally, one can even prove that $Y$ admits a (IV)-polyhedral norm, \cite{DDRS}.) Thus, Theorem~\ref{thm: tiling iff VI+Delta} gives the result.
\end{proof}

We now give an example where the cardinality of the locally finite tiling is less than the density character of the space. Notice that examples of non-separable normed spaces that admit a countable tiling were already known (\cite{FPZ_BLMS}*{Proposition~1.6}), but the tiling in our example has the additional feature of being locally finite. Note that such an example is not possible in a Banach space (Fact~\ref{fact: card of tiling in Banach}).

\begin{example}\label{ex: ell_Infty^F} Consider the normed space $\ell_\infty^F$ of all sequences that attain finitely many values (the notation $\ell_\infty^F$ is as in \cite{DHR_JMAA}), namely
\[ \ell_\infty^F\coloneqq \{ (x_n)_{n=1}^\infty\in \ell_\infty \colon |x(\N)|<\omega\}. \]
Plainly, this space has a countable boundary, given by $\B_0= \{\pm f_n\}_{n=1}^\infty$, where $f_n\in \ell_\infty^*$ is defined by $f_n(x)= x_n$. Thus, $\ell_\infty^F$ is a non-separable normed space and yet it admits a countable locally finite tiling by Corollary~\ref{cor: ctble bdry -> tiling}.

Incidentally, notice that each functional $\pm f_n$ defines a true face, as $\pm e_n$ belongs to $\inte_{H_n} (B_{\ell_\infty^F}\cap H_n)$, where $H_n= f_n^{-1}(1)$. In fact, if $y\in H_n$ and $\|e_n- y\|<1$, then $y\in B_{\ell_\infty^F}$. As a consequence, $\B_0$ is actually the minimal boundary and the sphere of $\ell_\infty^F$ is covered by countably many true faces. Further, as $\B_0$ is countable, it is clear that $\cconv(\B_0)\neq B_{\ell_\infty^*}$, thus the formula $\cconv(\B_0)= B_{X^*}$ from the structural theorem for polyhedral Banach spaces also doesn't extend to normed spaces.

Finally, we also observe that $\ell_\infty^F$ is (V)-polyhedral and it doesn't have property ($\Delta$). In fact, it is immediate to verify that the boundary $\B_0$ satisfies (V), thus $\ell_\infty^F$ is (V)-polyhedral by Lemma~\ref{lemma: V poly for bdry}. On the other hand, if $x\in \ell_\infty^F$ attains its maximum at infinitely many values, \emph{e.g.}, $x_n=1$ for all $n\in \N$, then $\D(x)$ contains infinitely many elements of $\B_0$. As $\B_0$ is the minimal boundary, all its elements are extreme points of $B_{\ell_\infty^*}$. Thus, $\ext \D(x)$ is infinite and $\ell_\infty^F$ does not have ($\Delta$).
\end{example}

\begin{fact}\label{fact: card of tiling in Banach} Every locally finite tiling of a Banach space $X$ by bounded convex bodies has cardinality $\textup{dens}(X)$.
\end{fact}

\begin{proof} Let $\cal T$ be a locally finite tiling of a Banach space $X$, and let $\kappa\coloneqq |\cal T|$. Clearly, $\kappa\leq \textup{dens}(X)$. For the converse inequality, note that the set $\cal F$ in Step~\ref{step: tiling -> VI+Delta} of Theorem~\ref{thm: tiling iff VI+Delta} also has cardinality $\kappa$. Thus, up to replacing $\cal T$ with $\cal F$ and taking an equivalent norm, we assume that $B_X\in \cal T$. For each $C\in \cal T$, $C\neq B_X$, take a functional $f_C\in S_{X^*}$ that separates $C$ from $B_X$, and let $\B\coloneqq \{f_C\colon C\in\cal T,\ C\neq B_X\}$. Then, $|\B|\leq \kappa$. Further, $\B$ is a boundary for $X$. In fact, for each $x\in S_X$, Remark~\ref{rmk: loc finite tiling}\ref{i: regular} implies that there is $C\in \cal T$, $C\neq B_X$ such that $x\in C$. In turn, this gives that $f$ supports $B_X$ at $x$, hence $\B$ is a boundary. 

However, by Corollary~\ref{cor: locally finite gives Delta QP} and Lemma~\ref{lemma: Delta-QP pts iff VI+Delta}, $X$ is a polyhedral Banach space. Thus, it has a minimal boundary, whose cardinality equals $\textup{dens}(X)$, by \cite{fonfstruttnew}*{Theorem 1.4} (see also \cite{veselystrutt}*{Theorem 2}). Hence, $\textup{dens}(X)\leq |\B|\leq \kappa$, and we are done.
\end{proof}

\section{Counterexamples}\label{sec: counterexamples}
In this section we give several counterexamples in different directions, with the aim of showing that the versions of the structural theorem that we proved in Section~\ref{sec: structural thm} are optimal. In particular, we give examples of polyhedral normed spaces that don't admit any algebraic true face (Section~\ref{sec: ell1,0} and Section~\ref{sec: countable dim}), examples where the algebraic true faces cover the sphere, but the true faces do not cover it (Section~\ref{sec: K+Delta no ST}), and one example of a polyhedral normed space $X$ such that the extreme points of $B_X$ are dense in $S_X$ (Section~\ref{sec: str exp dense}). We also give some examples involving boundaries (Section~\ref{sec: boundary}).

\subsection{A basic example, the space \texorpdfstring{$\ell_{1,0}$}{ell1,0}}\label{sec: ell1,0}
We begin with the simplest instance of a polyhedral normed space that doesn't satisfy the conclusion of the structural theorem, namely the space $\ell_{1,0} \coloneqq (c_{00},\n_1)$. The fact that this space is (VI)-polyhedral has already been noted, \emph{e.g.}, in \cite{DP-Modena}*{Example~1, p.~642}. As it turns out, $\ell_{1,0}$ is (V)-polyhedral and its unit ball doesn't even have any algebraic true face.

\begin{example}\label{ex: (V) no alg true faces} Consider the normed space $\ell_{1,0} \coloneqq (c_{00},\n_1)$. To begin with, notice that $\ell_{1,0}$ is (K)-polyhedral: in fact, each finite-dimensional subspace of $\ell_{1,0}$ is contained in $\ell_1^n$ for some $n\in \N$, and $\ell_1^n$ is clearly polyhedral. On the other hand, a standard computation (see, \emph{e.g.}, \cite{FHHMZ}*{Exercise 7.37}) shows that the $\ell_1$-norm is G\^{a}teaux differentiable at a point $x\in S_{\ell_1}$ if and only if all coordinates of $x$ are non-zero. Hence, the norm of $\ell_{1,0}$ is nowhere G\^{a}teaux differentiable, and Proposition~\ref{prop: poly VS smooth} yields that $\ell_{1,0}$ admits no algebraic true face. In particular, by Theorem~\ref{thm: structural in normed space}, the space does not have ($\Delta$), which is also easy to check directly.

As it turns out, $\ell_{1,0}$ is even (V)-polyhedral. In fact, fix $x\in S_{\ell_{1,0}}$ and consider the finite set $\supp(x)\coloneqq \{j\in \N\colon x_j\neq 0\}$. Then,
\[ \D(x)= \{f\in \ell_\infty\colon f_j= \sgn(x_j),\, \text{for all}\, j\in \supp(x) \}. \]
Since the extreme points of the dual ball $B_{\ell_ \infty}$ are the sequences with values in $\{-1,1\}$, if $f\in \ext B_{\ell_\infty}\setminus \D(x)$, there must be $k\in \supp(x)$ such that $f_k= -\sgn (x_k)$. Therefore, setting $\e\coloneqq \min\{|x_j|\colon j\in \supp(x)\}>0$, we have that 
\[ f(x)\leq 1- 2|x_k|\leq 1-2\e. \]
Hence, for each $x\in S_{\ell_{1,0}}$, we obtain
\[ \sup \{f(x)\colon f\in \ext B_{\ell_\infty} \setminus \D(x)\}<1, \]
which shows that $\ell_{1,0}$ is (V)-polyhedral.

The last property we wish to observe in this example is the equality
\[ \conv (\ext B_{\ell_{1,0}})= B_{\ell_{1,0}} \]
(notice that we are only taking the convex hull and not its closure), which immediately follows from the fact that $\ext B_{\ell_{1,0}}= \{\pm e_n\}_{n\in \N}$.
\end{example}

The takeaway of this example is that if one weakens the assumption of (IV)-polyhedrality in Corollary~\ref{cor: str thm for IV-poly} to (V)-polyhedrality, not only the conclusion ceases to hold, but it is even possible that there is no algebraic true face at all. In Section~\ref{sec: countable dim} we will obtain the same phenomenon in a suitable renorming of every normed space of countable dimension.
\smallskip

By modifying the previous example, we can also obtain another one where, instead, there are no extreme points at all.

\begin{example} Consider the normed space $X\coloneqq \ell_{1,0}\oplus_\infty c_0$ (one could also consider $\ell_{1,0}\oplus_\infty c_{00}$). Since the unit ball of $c_0$ doesn't have extreme points, one quickly deduces that $\ext B_X= \emptyset$. Moreover, $S_X$ is not covered by the algebraic true faces. In fact, take any point $(x,y)\in S_X$ such that $\|y\|< 1$. Then, any functional $(f,g)\in \D_X((x,y))$ must have the form $(f,0)$, where $f\in \D_{\ell_{1,0}}(x)$. Hence,
\[ (f,0)^{-1}(1)\cap B_X= \left(f^{-1}(1)\cap B_{\ell_{1,0}}\right) \times B_{c_0}. \]
By the previous example, $f^{-1}(1)\cap B_{\ell_{1,0}}$ has empty algebraic interior in $S_{\ell_{1,0}}$, thus $(x,y)$ does not belong to any algebraic true face.

Incidentally, let us notice that points of the form $(x,y)\in S_X$ with $\|y\|=1$ belong to a true face. In fact, find $k\in \N$ such that $|y_k|=1$; for simplicity, assume that $y_k=1$. Then, $(x,y)\in (0,e^*_k)^{-1}(1)$ and the face 
\[ (0,e^*_k)^{-1}(1)\cap B_X \]
has non-empty interior in $S_X$; for instance, $(0,e_k)$ belongs to the interior.
\end{example}

The rationale behind this example is that, for a (K)-polyhedral normed space $X$, property ($\Delta$) implies both the validity of the structural theorem and the fact that $\ext B_X= \emptyset$ (the latter fact is because of Lemma~\ref{lemma: few extreme points}). One might thus wonder if the condition $\ext B_X= \emptyset$ could already imply the validity of the structural theorem; this example shows that it is not the case. In \cite{DDRS} we will even give an example of a polyhedral normed space $X$ such that $\ext B_X= \emptyset$ and $S_X$ doesn't admit any algebraic true face.

\subsection{(K)-polyhedral spaces with property (\texorpdfstring{$\Delta$}{Δ})}\label{sec: K+Delta no ST}
The goal of this section is to give examples showing that (K)-polyhedrality together with property ($\Delta$) is not sufficient to obtain the validity of the structural theorem. In particular, it is possible that the unit sphere of a polyhedral space is covered by algebraic true faces, but not by true faces. Along the way, we also show that the implications \ref{item: G-smooth}$\implies$\ref{item: frechet}$\implies$\ref{item: true face} in Proposition~\ref{prop: poly VS smooth} do not hold in general.

The two examples we give are based on the same pattern and we isolate the general framework in the following proposition.

\begin{proposition}\label{prop: c_00 renorming} Given two sequences $(a_n)_{n=2}^{\infty}, (b_n)_{n=2}^{\infty}\subseteq (0,1]$ with $\inf_{n\geq 2}\frac{a_n}{b_n}= 0$, consider the functionals $h_n$ ($n\geq 2$) on $c_{00}$ given by
\[ h_n\coloneqq \frac{e_1^*+ b_ne_n^*}{1+a_n}, \]
where $(e_n^*)_{n=1}^{\infty}$ is the canonical Schauder basis of $\ell_1$. Consider the set
\[ \B \coloneqq \{\pm e_n^*\colon n \in \N \} \cup \{\pm h_n\colon  n \geq 2 \} \]
and the equivalent norm $\n$ on $c_{00}$ given by
\[ \|x\|\coloneqq \sup_{f\in \B}f(x). \]
Then, the space $X\coloneqq (c_{00},\n)$ satisfies the following: 
\begin{enumerate}
    \item\label{i: K+Delta} $X$ is \textup{(K)}-polyhedral and has property \textup{($\Delta$)};
    \item\label{i: B boundary} the set $\B$ is a boundary for $X$, thus $B_{X^*}=\cconv^{w^*}(\B)$;
    \item\label{i: e*1 not true face} $(e_1^*)^{-1}(1)\cap B_X$ is not a true face;
    \item\label{i: faces don't cover} $\D(e_1)= \{e^*_1\}$ and $e_1$ does not belong to any true face.
\end{enumerate}
\end{proposition}
	
\begin{proof} To begin with, $\n$ is clearly a norm on $c_{00}$, equivalent to $\n_\infty$. Moreover, by definition $\B\subseteq B_{X^*}$. The crucial observation that most of the proof depends on is the following one. If $x\in X$ and $N\coloneqq \max (\supp(x))$, then for all $n>N$ we have
\begin{equation}\label{eq: large h_n don't count}
    |h_n (x)|= \frac{|e_1^*(x) + b_ne_n^*(x)|}{1+a_n}= \frac{|e_1^*(x)|}{1+a_n}\leq \frac{\|x\|}{1+a_n}< \|x\| \quad\text{and}\quad e^*_n(x)=0.
\end{equation}
As a consequence, we obtain
\begin{equation}\label{eq: norm given by 1...N}
    \|x\|=\max\left\{ \max_{1\leq n\leq N}|e^*_n(x)|, \max_{2\leq n\leq N}|h_n(x)| \right\}.
\end{equation}
Therefore, $\B$ is a boundary, which gives that $B_{X^*}= \cconv^{w^*}(\B)$, and proves \ref{i: B boundary}.

Next, to show that $X$ is (K)-polyhedral, let $Z\coloneqq \spann\{z_1, z_2\}\subseteq X$ and let $N\coloneqq \max (\supp(z_1) \cup \supp(z_2))$. Thus, for all $x\in Z$ we have $\max(\supp(x))\leq N$, so \eqref{eq: norm given by 1...N} holds for the same index $N$ for all $x\in Z$. Hence, $B_Z$ is a polytope, and $X$ is (K)-polyhedral. Moreover, \eqref{eq: large h_n don't count} also implies that $\D(x)\cap \B$ is finite for all $x\in S_X$. As $\B$ is a boundary, Fact~\ref{fact: Delta iff B cap D finite} implies that $X$ has property ($\Delta$), and \ref{i: K+Delta} is also proved.

For \ref{i: e*1 not true face}, let $H\coloneqq (e_1^*)^{-1}(1)$ and suppose by contradiction that $H\cap B_X$ is a true face of $B_X$. Then, there exists $x\in \inte_H(B_X\cap H)$, thus we can find $r>0$ such that all $y\in H$ with $\|x-y\|\leq r$ belong to $B_X$. Since it is easily checked that $\|e_n\|=1$ for all $n$, it follows in particular that $x+re_n\in B_X$ for all $n\geq 2$. As a consequence, since $x+r e_n\in H$, for all $n> N\coloneqq \max(\supp(x))$ we obtain that
\[ |h_n(x + r e_n)|= \frac{|(e^*_1+ b_n e^*_n) (x+ re_n)|}{1+a_n}= \frac{1+ b_n r}{1+a_n} \leq 1, \]
whence
\[ 0<r \leq \frac{a_n}{b_n}, \]
for every $n > N$. However, this contradicts the assumption that $\inf_{n\geq 2}\frac{a_n}{b_n}= 0$.

Finally, for \ref{i: faces don't cover}, \eqref{eq: large h_n don't count} implies that $\D(e_1)\cap \B= \{e^*_1\}$ (recall that $h_n$ is only defined for $n\geq 2$). Thus, Lemma~\ref{lemma: boundary VS extreme points} gives $\D(e_1)= \{e^*_1\}$. Hence, if a true face contains $e_1$, it must be the one induced by $e^*_1$. But $(e_1^*)^{-1}(1)\cap B_X$ is not a true face, by \ref{i: e*1 not true face}.
\end{proof}

\begin{rem} The fact that $\D(e_1)= \{e^*_1\}$ can be given an alternative proof, by showing directly that $e_1$ is in the algebraic interior of $(e_1^*)^{-1}(1)\cap B_X$. Since this argument is simple and possibly instructive, we present it here. We aim to show that for all $y\in \ker e^*_1$ there exists $r>0$ such that $\|e_1+ ty\|\leq 1$ for all $|t|\leq r$. Assume without loss of generality that $\|y\|_\infty=1$. Plainly, $e^*_1(e_1+ ty)= 1$, while for $n\geq 2$, $|e^*_n(e_1+ ty)|= |ty_n|\leq r$. Further, for $n\geq 2$,
\[ |h_n(e_1+ ty)|= \left| \frac{1+ tb_ny_n}{1+a_n} \right|\leq \frac{1+ r b_n |y_n|}{1+a_n}. \]
When $n\notin \supp(y)$, this quantity is less than $1$; instead, for $n\in \supp(y)$, it is at most $1$ if and only if $r\leq \frac{a_n}{b_n|y_n|}$. Thus, setting
\[ r\coloneqq \min \left\{1, \frac{a_n}{b_n|y_n|}\colon n\in \supp(y)\right\} \]
does the job (recall that $\supp(y)$ is a finite set). Incidentally, this computation also shows that if we instead assumed that $\frac{a_n}{b_n}\geq \delta>0$ for all $n\geq 2$, then $e_1$ would belong to the interior of $(e_1^*)^{-1}(1)\cap B_X$.
\end{rem} 

\begin{example}\label{ex: (K) + (Delta) not (VI)} Take a sequence $(\e_n)_{n=2}^{\infty}\subseteq (0,1)$ such that $\e_n\to 0$ and apply the construction in Proposition \ref{prop: c_00 renorming} with $a_n= \e_n$ and $b_n=1$ for all $n\geq 2$. In particular, we have
\[ h_n\coloneqq \frac{e_1^*+ e_n^*}{1+\e_n}. \]
Then, the resulting space $X$ is a renorming of $c_{00}$ which is (K)-polyhedral, has property ($\Delta$), and yet $S_X$ is not covered by true faces. On the other hand, Theorem~\ref{thm: structural in normed space} implies that $S_X$ is covered by algebraic true faces. In particular, it follows that $X$ is not (VI)-polyhedral, thus (K)-polyhedral normed spaces with ($\Delta$) need not be (VI)-polyhedral. (However, recall that the implication holds for Banach spaces, \cite{InfPoly}*{Theorem 3.6}.)

We can also use this example to show that \ref{item: G-smooth} does not imply \ref{item: frechet} in Proposition~\ref{prop: poly VS smooth}. In fact, by Proposition~\ref{prop: c_00 renorming}\ref{i: faces don't cover}, $\D(e_1)=\{e_1^*\}$, so the point $e_1$ is a G\^ateaux smooth point in $B_X$. Still, it is not a Fr\'echet smooth point. In fact, $h_n(e_1)=\frac{1}{1+\e_n}\to 1$ as $n\to\infty$ and yet $(h_n)_{n=1}^{\infty}$ does not converge to $e^*_1$ (in norm). Thus, by \v{S}mulyan Lemma, $e_1$ is not a Fr\'echet smooth point. 
\end{example}

The next example is a variation of the preceding one and it aims to show that \ref{item: frechet} does not imply \ref{item: true face} in Proposition~\ref{prop: poly VS smooth}. In the example we will need the following folklore lemma, very similar, \emph{e.g.}, to \cite{Bourgin}*{Lemma 3.5.1}, \cite{ChoiJung}*{Lemma 2.1}, or \cite{LipDaugavet}*{Lemma 5.5}. Recall that (\emph{e.g.}, \cite{FHHMZ}*{Definition 7.10}), given a set $A\subseteq X$, a point $p\in \overline{A}$, and $f\in X^*$, $f$ \emph{strongly exposes $p$ in $A$} if for each sequence $(x_n)_{n= 1}^\infty$ in $A$
\[ f(x_n)\to \sup_{x\in A}f(x)\quad \text{implies}\quad x_n\to p. \]

\begin{lemma}\label{lemma: strongly exposed convex hull} Let $A\subseteq X$ be a bounded set, $p\in \overline{A}$, and $f\in X^*$. Suppose that $f$ strongly exposes $p$ in $A$, then $f$ strongly exposes $p$ in $\overline{\conv}(A)$.
\end{lemma}

\begin{proof} We first show that if $f$ strongly exposes $p$ in $\conv(A)$, then $f$ strongly exposes $p$ in $\overline\conv(A)$. Let $(x_n)_{n= 1}^\infty\subseteq \overline\conv(A)$ be such that 
\[ f(x_n)\to \sup_{x\in A} f(x)=f(p). \]
For each $n\in\N$ take $y_n\in\conv(A)$ such that $\|y_n-x_n\|<\frac1n$. Hence,
\begin{equation*}
    |f(y_n)-f(p)|\leq \|f\|\cdot\|y_n-x_n\| + |f(x_n)-f(p)|\to 0.
\end{equation*}
Since $f$ strongly exposes $p$ in $\conv(A)$, we have $y_n\to p$; thus $x_n\to p$ as well. 

It remains to show that if $f$ strongly exposes $p$ in $A$, then $f$ also strongly exposes $p$ in $\conv(A)$. Without loss of generality we assume that $p=0$. In particular, $f(a)\leq 0$ for all $a\in A$. Let $(x_n)_{n=1}^\infty\subseteq \conv(A)$ be such that $f(x_n)\to f(p)= 0$. For each $n$, there exists $\{\lambda_{n,a}\}_{a\in A}\subseteq [0,1]$ such that
\[ x_n=\sum_{a\in A}\lambda_{n,a} a, \quad \sum_{a\in A}\lambda_{n,a}=1, \]
and, for each $n\in \N$, only finitely many $\lambda_{n,a}$'s are non-zero. Fix $\e>0$. Since $f$ strongly exposes $0$ in $A$, there exists $\alpha<0$ such that $\|a\|\leq \e$ for all $a\in A$ with $f(a)\geq \alpha$. Let $I\coloneqq \{a\in A\colon f(a)\geq \alpha\}$. We claim that
\begin{equation}\label{eq: sum lambdas to 0}
    \sum_{a\in A\setminus I}\lambda_{n,a} \to 0, \quad \text{as } n\to \infty.    
\end{equation}
Indeed,
\begin{equation*}
    f(x_n)= \sum_{a\in I} \underbrace{\lambda_{n,a} f(a)}_{\leq 0} + \sum_{a\in A\setminus I } \lambda_{n,a} f(a) \leq \alpha \sum_{a\in A\setminus I} \lambda_{n,a}\leq 0.
\end{equation*}
Taking the limit for $n\to \infty$, we get $ \lim_{n\to \infty} \alpha\sum_{a\in A\setminus I} \lambda_{n,a}= 0$. Since $\alpha\neq 0$, the claim is proved. Finally, we show that $x_n\to 0$. As $A$ is bounded, $M\coloneqq \sup_{a\in A}\|a\|< \infty$. Thus, we have (recall that $\|a\|\leq \e$ when $a\in I$)
\begin{equation*}
    \|x_n\|\leq \sum_{a\in I}\lambda_{n,a}\|a\| + \sum_{a\in A\setminus I}\lambda_{n,a}\|a\|  \leq \e \sum_{a\in I}\lambda_{n,a} + M\sum_{a\in A\setminus I}\lambda_{n,a}\leq \e+ M\sum_{a\in A\setminus I}\lambda_{n,a}.
\end{equation*}
Letting $n\to \infty$, by \eqref{eq: sum lambdas to 0} we get 
\begin{equation*}
    0\leq \limsup_{n\to \infty}\|x_n\|\leq \e.
\end{equation*}
Since $\e>0$ was arbitrary, we are done.
\end{proof}
    
\begin{example}\label{ex: Frechet vs true face} The goal of this example is to construct an equivalent polyhedral norm on $c_{00}$ with a Fr\'echet smooth point which does not belong to the relative interior of any true face. This shows that in general \ref{item: frechet} does not imply \ref{item: true face} in Proposition~\ref{prop: poly VS smooth}. Take a sequence $(\e_n)_{n=2}^{\infty}\subseteq (0,1)$ such that $\e_n\to 0$ as $n\to \infty$ and consider, for each $n\geq 2$, the functionals 
\[ h_n\coloneqq \frac{e_1^* + \sqrt{\e_n} e_n^*}{1+\e_n} \]
(namely, we choose $a_n= \e_n$ and $b_n= \sqrt{\e_n}$ in Proposition~\ref{prop: c_00 renorming}). Then, the space $X$ constructed in Proposition~\ref{prop: c_00 renorming} is (K)-polyhedral, has ($\Delta$), and $e_1$ does not belong to any true face. However, we show that $e_1$ is a Fr\'echet smooth point in $B_X$. Again by \v{S}mulyan lemma, this amounts to proving that if a sequence $(g_n)_{n=1}^\infty$ in $B_{X^*}$ satisfies $g_n(e_1)\to 1$, then $g_n\to e^*_1$ (in the norm of $X^*$). In other words, $e_1$ (seen in $X^{**}$) strongly exposes $e^*_1$ in $B_{X^*}$.

For this, we start by observing that if $\cal A\coloneqq \{\pm h_n\}_{n\geq 2}$,  by Proposition~\ref{prop: c_00 renorming}\ref{i: B boundary}
\[ B_{X^*}= \cconv^{w^*}(\B)= \cconv^{w^*}\left( B_{(\ell_1,\n_1)}\cup \cconv(\cal A) \right). \]
However, the sequence $(h_n)_{n= 2}^\infty$ converges in norm to $e_1^*$. Thus, $\cconv(\cal A)$ is norm compact (so, $w^*$-compact as well). As $B_{(\ell_1,\n_1)}$ is also $w^*$-compact, we obtain
\[ B_{X^*}= \conv\left( B_{(\ell_1,\n_1)}\cup \cconv(\cal A) \right). \]

We now use this equality and Lemma~\ref{lemma: strongly exposed convex hull} to show that $e_1$ strongly exposes $e_1^*\in B_{X^*}$. First, we prove that $e_1$ strongly exposes $e_1^*$ in $\cal A$. Given any sequence $(g_n)_{n=1}^{\infty}$ in $\cal A$ such that $g_n(e_1)\to 1$, take $k_n\in\N$ such that $g_n= h_{k_n}$. Hence, we have
\[ g_n(e_1)=\frac{1}{1+\e_{k_n}}\to 1, \]
so the sequence $(k_n)_{n=1}^{\infty}$ goes to infinity. Since, as observed above $h_n\to e^*_1$, we get that $g_n$ converges to $e_1^*$. Thus, $e_1$ strongly exposes $e_1^*$ in $\cal A$ and Lemma~\ref{lemma: strongly exposed convex hull} implies that $e_1$ strongly exposes $e_1^*$ in $\cconv(\cal A)$. 

Moreover, $e_1$ clearly strongly exposes $e_1^*$ in $B_{(\ell_1,\n_1)}$, thus it also strongly exposes $e_1^*\in B_{\ell_1}\cup \cconv(\cal A)$. Finally, a second application of Lemma~\ref{lemma: strongly exposed convex hull} yields that $e_1$ strongly exposes $e_1^*\in \conv(B_{\ell_1}\cup \cal A)= B_{X^*}$.

To conclude this example, let us observe that the boundary $\B$ actually coincides with $\wstr(B_{X^*})$. In fact, using Lemma~\ref{lemma: strongly exposed convex hull} as above it is immediate to check that $e_n$ strongly exposes $e^*_n$ and that $\frac{1+\e_n}{1+ \sqrt{\e_n}}(e_1+ e_n)$ strongly exposes $h_n$ for all $n\geq 2$.  Thus, each element of $\B$ is a $w^*$-strongly exposed point in $B_{X^*}$, and the fact that $\wexp(B_{X^*})$ is contained in each boundary implies that $\B= \wstr(B_{X^*})$. Hence, $\wstr(B_{X^*})$ being a boundary does not imply that $S_X$ is covered by true faces (see Remark~\ref{rmk: str exp vs true faces}).
\end{example}

\begin{rem} We conclude this section with one remark on Lemma~\ref{lemma: strongly exposed convex hull}, related to Lipschitz-free spaces. It is proved in \cite{LipDaugavet}*{Theorem 5.4} that a molecule $\frac{\delta_x- \delta_y}{d(x,y)}$ is a strongly exposed point of $B_{\cal F(M)}$ if and only if there is a function $f\in \textup{Lip}_0(M)$ that peaks at the pair $(x,y)$. From the definition of peaking function (see, \emph{e.g.}, \cite{LipDaugavet}*{Definition 5.2}, or \cite{Weaver}*{Definition 3.30}) it is clear that $f$ peaks at the pair $(x,y)$ if and only if it strongly exposes the molecule $\frac{\delta_x- \delta_y}{d(x,y)}$ in the set $\textup{Mol}(M)$ of all molecules on $M$. Since $\cconv(\textup{Mol}(M))= B_{\cal F(M)}$ by \cite{Weaver}*{Proposition 3.29}, the result in \cite{LipDaugavet} is a particular case of Lemma~\ref{lemma: strongly exposed convex hull}.
\end{rem}

\subsection{(V)-polyhedral normed spaces without algebraic true faces} \label{sec: countable dim}
In this part we generalise the result of Section~\ref{sec: ell1,0} to every normed space of countable dimension. In particular, this gives a (V)-polyhedral renorming of $c_{00}$ with no algebraic true faces.

\begin{theorem}\label{thm: V no true faces} Let $X$ be any normed space with countable dimension. Then, there exists a \textup{(V)}-polyhedral renorming $\nn$ of $X$ with no algebraic true faces. Further, the unit ball of $(X,\nn)$ is the convex hull of its extreme points.
\end{theorem}

\begin{rem}\label{rmk: V iff 1 not cluster} Let us note here the following characterisation of (V)-polyhedrality, that we use in the proof. A normed space $X$ is (V)-polyhedral if and only if there exists a $1$-norming set $\cal C$ such that, for each $x\in S_X$, $1$ is not a cluster point of the set $\{f(x)\colon f\in \cal C\}$. In fact, this condition is clearly equivalent to
\[ \sup \{f(x)\colon f\in \cal C,\ f(x)< 1\}= \sup \{f(x)\colon f\in \cal C \setminus \D(x)\} <1, \]
which is equivalent to (V)-polyhedrality, by Lemma~\ref{lemma: V poly for bdry}.
\end{rem}

\begin{proof} According to Lemma~\ref{lemma: ctble dim => IV}, there exists norm $\n$ on $X$ that is (IV)-polyhedral. In particular, by Fact~\ref{fact: IV iff D + V}, $(X,\n)$ is (V)-polyhedral and has ($\Delta$). By Markushevich's result \cite{Markushevich} that we recalled in Section~\ref{sec: prelim}, we can write $X=\spann\{e_n\}_{n=1}^\infty$, for some biorthogonal system $\{e_n; e^*_n\}_{n= 1}^\infty$; further, we can assume that $\|e^*_n\| =1$. Let $X_n\coloneqq \spann\{e_1, \dots, e_n\}$. Moreover, we let $\B_0$ be the minimal boundary of $(X,\n)$. We first notice that, for all finite-dimensional subspaces $Y$ of $X$, the set
\begin{equation}\label{eq: norming X_n finite set}
    \{f\in \B_0\colon f(x)=1 \text{ for some } x\in S_Y\} \quad \text{is a finite set.}
\end{equation}
In fact, suppose on the contrary that the set is infinite; then, there are an injective sequence $(f_n)_{n=1}^\infty\subseteq \B_0$ and a sequence $(y_n)_{n=1}^\infty\subseteq S_Y$ such that $f_n(y_n) =1$. We can also assume that $y_n\to y\in S_Y$. Let $f\in \bigl(\{f_n\colon n\in\N\}\bigr)'\subseteq \B_0'\subseteq (\ext B_{X^*})'$; then clearly $f(y)=1$, which contradicts property (IV). Thus, \eqref{eq: norming X_n finite set} is proved.

We now define the set $\cal C$ that will be a boundary for the desired renorming. Fix $\e>0$ and a sequence $(\e_n)_{n=1}^{\infty}\subseteq (0,1]$ such that $\e_1=1$ and $\sum_{n=1}^{\infty}\e_n\leq 1+ \e$. Next, we define inductively the following subsets of $X^*$:
\[ \cal C_1 \coloneqq \{ f \in \B_0\colon f(x) = 1 \text{ for some } x \in S_{X_1} \} \]
and recursively for $n\geq 2$
\[ \B_n \coloneqq \{ f \pm \e_n e_n^*\colon f \in \cal C_{n-1} \}, \]
\[ \cal C_n\coloneqq \cal C_{n-1}\cup \B_n \cup \{ f \in \B_0\colon f(x) = 1 \text{ for some } x \in S_{X_n} \}. \]
Induction and \eqref{eq: norming X_n finite set} readily imply that each $\cal C_n$ is a finite set. Finally, we let $\cal C\coloneqq \bigcup_{n=1}^\infty \cal C_n$ and we define the desired renorming as 
\[ \vertt{x}\coloneqq \sup \{|f(x)| \colon f \in \cal C \}= \sup \{f(x)\colon f \in \cal C\}. \]

Our first task is to show that $\nn$ is an equivalent norm. First, as $\B_0$ is the minimal boundary for $(X,\n)$, for each $f\in \B_0$ there is $x\in S_{(X,\n)}$ with $f(x)=1$; taking $n\in \N$ such that $x\in S_{(X_n,\n)}$, we see that $f\in \cal C_n$. Thus, $\B_0\subseteq \cal C$. On the other hand, an immediate induction shows that each $f\in \cal C_n$ has norm at most $\sum_{j=1}^n \e_j$ (the step $n=1$ is valid because of our choice $\e_1=1$; also recall that $\|e^*_n\|=1$). Hence, $\|f\|\leq 1+ \e$ for all $f\in \cal C$. Combining these two facts gives
\begin{equation}\label{eq: nn close to n}
    \|x\|= \sup\{f(x)\colon f\in \B_0\}\leq \sup\{f(x)\colon f\in \cal C\}\leq (1+\e)\|x\|,
\end{equation}
whence $\nn$ is an equivalent norm. We shall split in steps the rest of the argument, showing that the renorming $(X,\nn)$ has the desired properties.

\begin{step} For all $x\in X_n$, we have
\begin{equation}\label{eq: values in C}
    \{f(x)\colon f\in \cal C\}= \{f(x)\colon f\in \cal C_n\cup \B_0\}.   
\end{equation}
\end{step}

For this, it is clearly enough to prove that for every $k\in \N_0$ the following condition holds: for all $g\in \cal C_{n+k}$ there exists $f\in \cal C_n \cup \B_0$ such that $f(x)=g(x)$. We proceed by induction, the base case $k=0$ being trivially satisfied. 

Suppose now that the condition is true for some $k\geq 0$, and let $g\in \cal C_{n+k+1}$. Then, $g$ belongs to either $\B_0$, or $\cal C_{n+k}$, or $\B_{n+k +1}$. In the first two cases, we are done (using the inductive assumption in the case that $g\in \cal C_{n+k}$). Instead, in the third case, by definition of $\B_{n+k +1}$, there is $h\in \cal C_{n+k}$ such that $g= h\pm\e_{n+k+1} e^*_{n+k+1}$. Thus, $g(x)= h(x)$, as $x\in X_n$. Applying the hypothesis of induction, there exists $f\in\cal C_n\cup \B_0$ such that
\[ f(x)=h(x)=g(x), \]
thus completing the inductive step. Hence, \eqref{eq: values in C} holds.

\begin{step}\label{step: 1 not cluster} $(X,\nn)$ is \textup{(V)}-polyhedral and, if $x\in S_{X_n}$, there is $f\in \cal C_n$ with $f(x)=1$.
\end{step}

For the first assertion, take $x\in S_{(X,\nn)}$ and find $n\in \N$ with $x\in X_n$. As $\cal C$ is by definition $1$-norming, the supremum of the sets in \eqref{eq: values in C} equals $1$. Recall that $\cal C_n$ is a finite set; hence, if $1$ is a cluster point of the sets in \eqref{eq: values in C}, it is also a cluster point of the set
\[ \{f(x)\colon f\in \B_0\}. \]
However, this implies that $x\in S_{(X,\n)}$ and, by Remark~\ref{rmk: V iff 1 not cluster}, that $(X,\n)$ is not (V)-polyhedral, a contradiction. Thus, $1$ is not a cluster point of the sets in \eqref{eq: values in C}, and a second application of Remark~\ref{rmk: V iff 1 not cluster} implies that $(X,\nn)$ is (V)-polyhedral.

For the second clause, take $x\in S_{X_n}$. By the previous argument, the supremum of the sets in \eqref{eq: values in C} equals $1$ and is not a cluster point, so it must be a maximum. Hence, there is $f\in \cal C_n\cup \B_0$ such that $f(x)=1$. In case $f\in \cal C_n$, we are done. On the other hand, if $f\in \B_0$, it is a functional in $\B_0$ that attains its norm on some point of $S_{X_n}$. By definition, $f$ is then among the functionals of $\B_0$ that are added to $\cal C_n$, hence $f\in \cal C_n$ in this case as well, and we are done.

\begin{step} $B_{(X,\nn)}$ has no algebraic true faces and $B_{(X, \nn)}=\conv(\ext B_{(X, \nn)})$.
\end{step}
For the first part, by Proposition~\ref{prop: poly VS smooth} we need to show that, for each $x\in S_{(X,\nn)}$, $\D_{(X,\nn)}(x)$ is not a singleton. Take any $x\in S_{(X,\nn)}$ and find $n\in \N$ with $x\in X_n$. By Step~\ref{step: 1 not cluster}, there exists $f\in \cal C_n$ with $f(x)=1$. By definition, $f\pm \e_{n+1}e^*_{n+1}\in \cal C_{n+1}$ and clearly both belong to $\D_{(X,\nn)}(x)$. Hence, $\D_{(X,\nn)}(x)$ is not a singleton. 

For the second one, since $X= \bigcup_{n=1}^\infty X_n$, we have
\[ B_{(X,\nn)}= \bigcup_{n= 1}^\infty B_{(X_n,\nn)}= \bigcup_{n= 1}^\infty \conv(\ext B_{(X_n,\nn)}) \subseteq\conv \left( \bigcup_{n=1}^\infty \ext B_{(X_n,\nn)} \right). \]
Thus, it is enough to prove that $\ext B_{(X_n, \nn)}\subseteq \ext B_{(X, \nn)}$, for each $n\in \N$. Fix $n\in \N$, take $x\in \ext B_{(X_n, \nn)}$ and suppose that $x= \frac{y+z}{2}$, where $y,z\in B_{(X, \nn)}$. We will show that $y,z\in X_n$, which, as $x\in \ext B_{(X_n, \nn)}$, then implies that $x= y= z$, thus yielding that $x\in \ext B_{(X, \nn)}$. In order to show that $y,z\in X_n$, use Step~\ref{step: 1 not cluster} to find $f\in \cal C_n$ such that $f(x)=1$; then, $f(y)= f(z)= 1$ as well. However, for each $k> n$, we have that $\cal C_n\subseteq \cal C_{k-1}$, whence $f\pm \e_k e^*_k\in \cal C_k\subseteq \cal C$. Thus, 
\[ |1\pm \e_k e^*_k(y)|= |(f\pm \e_k e^*_k)(y)|\leq 1, \]
which yields that $e^*_k(y)=0$, and the same for $z$. Hence, $e^*_k(y)= e^*_k(z)= 0$, for all $k>n$, which implies that $y,z\in X_n$, and concludes this step and the proof.
\end{proof}

\begin{rem} The proof even shows that every norm on $X$ can be approximated by a norm $\nn$ as in the conclusion of the theorem. In fact, by Lemma~\ref{lemma: ctble dim => IV} each norm on $X$ can be approximated by a (IV)-polyhedral norm $\n$. Then, \eqref{eq: nn close to n} shows that the norm $\nn$ can be taken to be arbitrarily close to $\n$.
\end{rem}

\subsection{Density of strongly exposed points}\label{sec: str exp dense}
The goal of this section is the construction of polyhedral normed spaces such that the set of strongly exposed points of the ball is dense in the unit sphere (in particular, this is true for the extreme points). Since this phenomenon is clearly impossible when the unit sphere is covered by true faces (in particular, in a polyhedral Banach space), this gives yet another illustration of the failure of the structural theorem. Notice that, by Lemma~\ref{lemma: few extreme points}, such an example cannot be (VI)-polyhedral and it cannot satisfy ($\Delta$).

\begin{theorem}\label{thm: extreme dense} Let $X$ be a normed space with countable dimension. Then $X$ admits a polyhedral renorming $\nn$ such that $\str(B_{(X,\nn)})$ is dense in $S_{(X,\nn)}$.
\end{theorem}

The rough idea of the proof is to inductively add strongly exposed points by taking the convex hull of a bounded convex body and a point outside it. The fact that the added point is strongly exposed is standard, we sketch a proof via Lemma~\ref{lemma: strongly exposed convex hull}.

\begin{fact}\label{fact: new str exp point} Let $X$ be a normed space and $C\subseteq X$ be a bounded convex body. If $x\in X\setminus C,$ then $x$ is a strongly exposed point of $\conv(\{x\}\cup C)$.
\end{fact}

\begin{proof} Take a functional $f$ such that $\sup_{y\in C}f(y) < f(x)$. Hence, $x$ is a strongly exposed point of $C\cup \{x\}$, and the assertion follows from Lemma~\ref{lemma: strongly exposed convex hull}. 
\end{proof}

\begin{proof}[Proof of Theorem~\ref{thm: extreme dense}] To begin with, we fix a sequence $(y_n)_{n=1}^\infty$ which is dense in $X$ and an algebraic basis for $X$. The existence of such a sequence follows, \emph{e.g.}, from \cite{BDHMP_bases}*{Proposition~3.2}, but the argument is so short that we recall it here. Fix a basis $(U_n)_{n=1}^\infty$ for the topology of $X$ and find inductively vectors $z_n\in U_n\setminus \spann\{z_1, \dots, z_{n-1}\}$. By construction, $(z_n)_{n=1}^\infty$ is dense and linearly independent, so it can be completed to a (countable) algebraic basis $(y_n)_{n=1}^\infty$ for $X$.

Further, by Lemma~\ref{lemma: ctble dim => IV} we can take a polyhedral norm $\n$ on $X$. Fix $\e\in(0,1)$, and choose a sequence $(\e_n)_{n= 1}^{\infty} \subseteq (0,1)$ with $1-\e\leq \prod_{n=1}^\infty(1- \e_n)$. For each $n\in \N$, we set $X_n\coloneqq \spann\{y_1,\dots, y_n\}$, so that $X= \bigcup_{n=1}^\infty X_n$, and $y_{n+1}\notin X_n$. We also set $X_0\coloneqq \{0\}$, and $P_0\coloneqq B_X$ (which is a polytope). We now define by induction a sequence $(P_n)_{n=1}^\infty$ of symmetric polytopes and points $x_n\in \R^+\{y_n\}$ such that, for all $n\in \N$:
\begin{enumerate}[(i)]
    \item\label{Pn: bodies close} $(1- \e_n)P_{n-1}\subseteq P_n\subseteq P_{n-1}$,
    \item\label{Pn: bodies extend} $P_{n-1}\cap X_{n-1}= P_n\cap X_{n-1}$,
    \item\label{Pn: str exp} $x_n\in \str(P_n)$.
\end{enumerate}

In the construction we will tacitly use several times the following standard fact (see, \emph{e.g.}, \cite{FHHMZ}*{Exercise~1.60}): if $C$ is a closed, bounded convex set, and $K$ is a compact convex set, then $\conv(C\cup K)$ is closed. For the case $n=1$, let $x_1$ be the only point in $\R^+\{y_1\}\cap \partial P_0$, so $x_1\notin (1- \e_1)P_0$. Thus, setting
\[ P_1\coloneqq \conv\big((1- \e_1)P_0\cup \{\pm x_1\} \big), \]
$P_1$ is closed and Fact~\ref{fact: new str exp point} shows that \ref{Pn: str exp} holds. By definition, also \ref{Pn: bodies close} holds, while \ref{Pn: bodies extend} is trivial, as $X_0=\{0\}$. It is also easy to check that $P_1$ is a symmetric polytope. Suppose now, that for some $n\in \N$ we have already constructed $(P_k)_{k=1}^n$ and $(x_k)_{k=1}^n$ with the above properties. As above, we let $x_{n+1}$ be the only point in $\R^+\{y_{n+1}\}\cap \partial P_n$. We then define (see Figure~\ref{FIG-set P})
\[ D_{n+1}\coloneqq \conv\big( (P_n\cap X_n)\cup (1-\e_{n+1})P_n \big). \]
Noting that $(1-\e_{n+1})P_n \subseteq \inte(P_n)$, we see that the unique points of $D_{n+1}$ that can belong to $\partial P_n$ are in $X_n$. Since $x_{n+1}\notin X_n$, we obtain that $x_{n+1}\notin D_{n+1}$. Thus, we can define
\begin{equation}\label{eq: def P n+1}
    P_{n+1}\coloneqq \conv\big( D_{n+1}\cup \{\pm x_{n+1}\} \big);
\end{equation}
once more, $P_{n+1}$ is closed and Fact~\ref{fact: new str exp point} yields that \ref{Pn: str exp} holds. By construction, it is also plain that \ref{Pn: bodies close} is satisfied. For \ref{Pn: bodies extend}, note that by definition $P_n\cap X_n\subseteq D_{n+1}\subseteq P_{n+1}$, giving one inclusion; the converse inclusion is clear, as $P_{n+1}\subseteq P_n$.

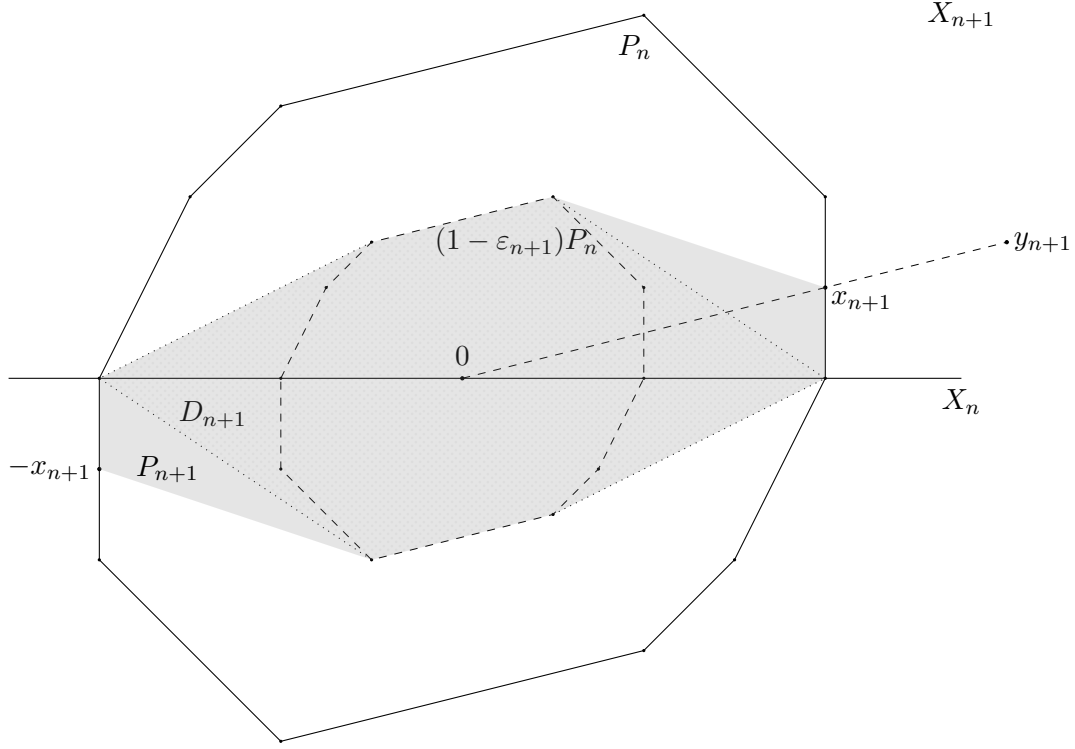
\begin{figure}
\centering	
\begin{tikzpicture}[scale=0.6]

	\draw[-](-10,0)--(11,0)node[below]{\small{\( X_n \)}};
	\node at (11,8) {\small{$X_{n+1}$}};
	\draw (0,0) circle[radius=1pt];
	\fill (0,0) circle[radius=1pt];
	\node at (0,.5) {\small{$0$}};

	\coordinate (A) at (-8,0);
	\coordinate (B) at (-6,4);
	\coordinate (C) at (-4,6);
	\coordinate (D) at (4,8);
	\coordinate (E) at (8,4);
	\coordinate (F) at (8,0);
	\coordinate (G) at (6,-4);
	\coordinate (H) at (4,-6);
	\coordinate (I) at (-4,-8);
	\coordinate (L) at (-8,-4);

	\draw[-](-8,0)--(-6,4)--(-4,6)--(4,8)--(8,4)--(8,0)--(6,-4)--(4,-6)--(-4,-8)--(-8,-4)--(-8,0);
	\node at (3.8,7.3) {\small{$P_{n}$}};
	\foreach \pongo in  {A,B,C,D,E,F,G,H,I,L} 
    \fill (\pongo) circle (1pt);

    \coordinate (M) at (-4,0);
	\coordinate (N) at (-3,2);
	\coordinate (O) at (-2,3);
	\coordinate (P) at (2,4);
	\coordinate (Q) at (4,2);
	\coordinate (R) at (4,0);
	\coordinate (S) at (3,-2);
	\coordinate (T) at (2,-3);
	\coordinate (U) at (-2,-4);
	\coordinate (V) at (-4,-2);

	\draw[dashed, thick, line width = 0.3 pt](-4,0)--(-3,2)--(-2,3)--(2,4)--(4,2)--(4,0)--(3,-2)--(2,-3)--(-2,-4)--(-4,-2)--(-4,0);
    \node at (1.2,3) {\small{$(1-\e_{n+1})P_{n}$}};
    \foreach \pongo in  {M,N,O,P,Q,R,S,T,U,V} 
    \fill (\pongo) circle (1pt);

	\draw (12,3) circle[radius=1pt];
	\fill (12,3) circle[radius=1pt];
	\node at (12.8,3) {\small{$y_{n+1}$}};
	\draw[dashed,thick, line width = 0.1 pt](-0,0)--(12,3);

	\draw (8,2) circle[radius=1pt];
	\fill (8,2) circle[radius=1pt];
	\node at (8.8,1.7) {\small{$x_{n+1}$}};	
	\draw (-8,-2) circle[radius=1pt];
	\fill (-8,-2) circle[radius=1pt];
	\node at (-9.1,-2) {\small{$-x_{n+1}$}};	

    \draw[dotted,thick, line width = 0.4 pt](-2,-4)--(-8,0);
	\draw[dotted,thick, line width = 0.4 pt](2,4)--(8,0);
    \draw[dotted,thick, line width = 0.4 pt](2,-3)--(8,0);
    \draw[dotted,thick, line width = 0.4 pt](-2,3)--(-8,0);
    \pattern[pattern=crosshatch dots, opacity=0.2] (-8,0)--(-2,3)--(2,4)--(8,0)--(2,-3)--(-2,-4)--(-8,0);
	\node at (-5.5,-.8) {\small{$D_{n+1}$}};

    \fill[gray, opacity = 0.2](-8,0)--(-2,3)--(2,4)--(8,2)--(8,0)--(2,-3)--(-2,-4)--(-8,-2)--(-8,0);
	\node at (-6.5,-2) {\small{$P_{n+1}$}};

\end{tikzpicture}
\caption{Proof of Theorem~\ref{thm: extreme dense}: in gray the set $P_{n+1}$.}\label{FIG-set P}
\end{figure}

Thus, the induction step is complete once we show that $P_{n+1}$ is a polytope. First, to show that $D_{n+1}$ is a polytope, fix a finite-dimensional subspace $Z$ of $X$ and assume without loss of generality that $X_n\subseteq Z$. Thus,
\[ D_{n+1}\cap Z= \conv \big((P_n\cap X_n)\cup (1-\e_{n+1})(P_n\cap Z) \big) \]
and both $P_n\cap X_n$ and $P_n\cap Z$ are finite-dimensional polytopes. Thus, $D_{n+1}$ is a polytope. The same argument applied to \eqref{eq: def P n+1} shows that $P_{n+1}$ is a polytope.

To conclude the proof, define $P\coloneqq \bigcap_{n=1}^\infty P_n$. Recalling that $P_0= B_X$, \ref{Pn: bodies close} implies that, for each $n\in \N$, $(1-\e)B_X\subseteq P_n\subseteq B_X$. Thus, $(1-\e)B_X\subseteq P\subseteq B_X$, and $P$ is a symmetric bounded convex body. Moreover, \ref{Pn: bodies extend} yields
\[ P_{n-1}\cap X_{n-1}= (P_n\cap X_n)\cap X_{n-1}= (P_{n+1}\cap X_n)\cap X_{n-1}= P_{n+1}\cap X_{n-1}; \]
continuing in the obvious way, we obtain that $P_n\cap X_n= P_k\cap X_n$ whenever $k\geq n$. Thus, as by \ref{Pn: bodies close} the sequence $(P_n)_{n=1}^\infty$ is decreasing, for each $n\in \N$ we have
\begin{equation}\label{eq: P cap Xn}
    P_n\cap X_n= P\cap X_n. 
\end{equation}
As each $P_n$ is a polytope, $P\cap X_n$ is a polytope for each $n$. Since each finite-dimensional subspace of $X$ is contained in some $X_n$, this implies that $P$ is a polytope. Thus, letting $\nn$ be the equivalent norm on $X$ whose unit ball is $P$, we have proved that $(X,\nn)$ is a polyhedral renorming of $X$.

Finally, recalling that $x_n\in X_n$ (because $\R^+\{y_n\}\subseteq X_n$), \eqref{eq: P cap Xn} implies that $x_n\in P$. In turn, \ref{Pn: str exp} and $P_n\subseteq P$ yield that $x_n\in \str(P)$ (use the same functional that strongly exposes $x_n$ in $P_n$).  Also, since $x_n\in \partial P$, the condition that $x_n\in \R^+\{y_n\}$ just means that $x_n= \frac{y_n}{\vertt{y_n}}$. The assumption that $(y_n)_{n=1}^\infty$ is dense in $X$ then easily implies that $(x_n)_{n=1}^\infty$ is dense in $S_{(X,\nn)}$, proving that $\str(B_{(X,\nn)})$ is dense in $S_{(X,\nn)}$.
\end{proof}

\begin{rem} The construction has a certain flexibility, because, for each sequence $(y_n)_{n=1}^\infty$ as at the beginning of the proof, the half-line $\R^+\{y_n\}$ contains a strongly exposed point. On the other hand, the assumptions on the sequence $(y_n)_{n=1}^\infty$ are necessary. In fact, if $\spann\{y_n\}_{n=1}^\infty\neq X$, one can only conclude that this span is polyhedral. Further, if $(y_n)_{n=1}^\infty$ is not linearly independent, it can happen that infinitely many $y_n$'s belong to some finite-dimensional subspace, forcing infinitely many extreme points in a finite-dimensional space, and destroying polyhedrality.
\end{rem}

\begin{rem} As is the case with Theorem~\ref{thm: V no true faces}, the above proof even gives the density of the norms $\nn$ with $\str(B_{(X,\nn)})$ dense in $S_{(X,\nn)}$.
\end{rem}

\subsection{Boundaries}\label{sec: boundary}
The last class of counterexamples we present involves boundaries of polyhedral spaces. For polyhedral Banach spaces, the structural theorem implies the existence of the minimal boundary which coincides with $\wexp(B_{X^*})$ and, in the separable case, the fact that this boundary is countable. We now show that in (separable) polyhedral normed spaces it is possible that there is no countable boundary, or that $\wexp(B_{X^*})$ is not a boundary. Most of what we say is a direct consequence of the previous material, and we just collect it in a single place here.

\begin{example}\label{ex: no ctble bdry} Consider the normed space $X$ of simple functions in $L_1$, namely $X\coloneqq\spann \{\bone_A\colon A\subseteq [0,1]\ \text{measurable}\}$. First, we observe that $X$ is polyhedral. In fact, if $Z$ is a finite-dimensional subspace of $X$, there exists a partition $\{A_1,\dots, A_n\}$ of $[0,1]$ (in measurable sets) such that $F\subseteq \spann \{\bone_{A_k}\colon k=1,\dots, n\}$. Since this linear span is isometric to $\ell_1^n$, which is polyhedral, we see that $X$ is polyhedral.

On the other hand, $X$ doesn't have a countable boundary. To see this, note that if a function $f\in X$ is a.e.~nonzero, then the only norming functional is (the equivalence class of) the function $\sgn (f)$. Thus, every boundary contains the set
\[ \B_0=\{\bone_A- \bone_{[0,1]\setminus A}\colon A\subseteq [0,1]\ \text{measurable} \}. \]
As $\B_0$ has cardinality continuum, we see that $X$ is a separable polyhedral space without a countable boundary. Incidentally, the set $\B_0$ coincides with $\ext (B_{L_\infty})$, so is a boundary. Hence, $\B_0$ is indeed the minimal boundary. Thus, Theorem~\ref{thm: bdry iff alg-ST} implies that $S_X$ is covered by (continuum many) algebraic true faces.

To conclude this example, we show that $X$ doesn't have ($\Delta$) and that its unit sphere doesn't admit any true face. By Proposition~\ref{prop: poly VS smooth}, it then follows that $X$ is not (VI)-polyhedral. To show that $X$ fails property ($\Delta$), take any function $f\in X$ such that $\{f=0\}$ has positive measure. Then, $\B_0\cap \D(f)$ is infinite. In fact, denoting Lebesgue's measure by $\lambda$, the function $[0,1] \ni t\mapsto \lambda (\{x\in [0,t]\colon f(x)=0\})$ is continuous and its value at $1$ is positive. Thus, there is a set $\Lambda\subseteq [0,1]$ of cardinality continuum such that the sets $[0,t]\cap \{f=0\}$ ($t\in \Lambda$) have mutually distinct measures. Hence, the functions
\[ \sgn(f)\cdot \bone_{\{f\neq 0\}}+ \left( \bone_{[0,t]}- \bone_{(t,1]} \right)\cdot \bone_{\{f=0\}} \qquad (t\in \Lambda) \]
are mutually distinct elements of $\B_0\cap \D(f)$. Since $\B_0$ is the minimal boundary, $\ext \D(f)= \ext B_{X^*}\cap \D(f)$ is also infinite. Thus, $X$ doesn't have ($\Delta$). 

For the second clause, suppose that a function $f$ belongs to the interior of some true face. Hence, all functions $g$ in $S_X$ that are sufficiently close to $f$ belong to the interior of the same true face. Thus, $\D(g)= \D(f)$ is a singleton. However, arbitrarily close to $f$ there are functions $g$ such that $\{g=0\}$ has positive measure, whence $\D(g)$ is infinite, a contradiction.
\end{example}

\begin{rem}\label{rmk: bdry of IV} In the opposite direction, recall from Example \ref{ex: ell_Infty^F} that the space $\ell_\infty^F$ is a non-separable polyhedral space that admits a countable boundary. In particular, by Fact~\ref{fact: ctble boundary and IV} it has a (IV)-polyhedral renorming with countable boundary. Thus, even for (IV)-polyhedral normed spaces, the minimal boundary can have cardinality smaller than the density character (compare with Fact~\ref{fact: bdry III}). In particular, the closed convex hull of the minimal boundary doesn't need to coincide with the dual ball (unlike the case for polyhedral Banach spaces).
\end{rem}

The second example is just a direct consequence of the results we already presented. In fact, by Theorem~\ref{thm: bdry iff alg-ST} the fact that $\wexp(B_{X^*})$ is a boundary is equivalent to the fact that $S_X$ is covered by algebraic true faces. Therefore, the space $X= \ell_{1,0}$ of Example~\ref{ex: (V) no alg true faces} is a (V)-polyhedral normed space for which $\wexp(B_{X^*})$ is not a boundary. By Fact~\ref{fact: bdry}, this also means that there is no boundary at all which is minimal with respect to inclusion.

\begin{rem} Finally, we wish to remark that, while $\wexp(B_{X^*})$ being a boundary implies that the unit sphere is covered by algebraic true faces, it does not imply the same with genuine true faces. In fact, the space in Example~\ref{ex: (K) + (Delta) not (VI)} is a (K)-polyhedral space with ($\Delta$), whose unit sphere is not covered by true faces. Yet, $\wexp(B_{X^*})$ is a boundary, because $S_X$ is covered by algebraic true faces, by Theorem~\ref{thm: structural in normed space}. It is even possible to give an example of a (K)-polyhedral space with ($\Delta$) whose unit sphere doesn't have any true face at all, \cite{DDRS}.
\end{rem}

\appendix
\section{Polyhedral normed spaces, a toolbox}\label{sec: appendix}
In this section we give a more detailed (compared to Section~\ref{sec: prelim}) overview of the results concerning polyhedrality that we need in our paper. For most of the results we also give a proof; when we do so, it is only because the proofs in the literature are given for Banach spaces, or use a significantly different notation.

\subsection{Equivalences of some definitions}

\begin{lemma}[\cite{FonfLindVes}*{Lemma~1.5}]\label{lemma: boundary VS extreme points} Let $X$ be a polyhedral normed space, $\B\subseteq B_{X^*}$ be a boundary for $X$, and $x\in S_X$. Then,
\[ \D(x)=\overline\conv^{w^*}[\D(x)\cap \B]. \]
In particular, $\D(x)=\conv[\D(x)\cap \B]$ whenever $\D(x)\cap \B$ is finite.
\end{lemma}

\begin{fact}[\cite{FonfLindVes}*{Remark~1.4(b)}]\label{fact: Delta iff B cap D finite} A polyhedral normed space $X$ has \textup{($\Delta$)} if and only if there exists a boundary $\B$ for $X$ such that $\D(x) \cap \B$ is finite for each $x\in S_X$.
\end{fact}

\begin{proof} Fact~\ref{fact: basic D(x)}\ref{i: ext D(x)=} implies that $\ext \D(x)= \D(x)\cap \ext B_{X^*}$, for each $x\in S_X$. Hence, if $X$ has property ($\Delta$), one can take $\B= \ext B_{X^*}$. Vice-versa, if there is a boundary $\B$ such that $\D(x) \cap \B$ is finite for each $x\in S_X$, Lemma~\ref{lemma: boundary VS extreme points} implies that $\D(x)$ is the convex hull of a finite set. Thus, it has finitely many extreme points, by Milman's `converse' of the Krein--Milman theorem (see, \emph{e.g.}, \cite{FHHMZ}*{Theorem 3.66}).
\end{proof}

Given a normed space $X$ and $\B\subseteq B_{X^*}$, $\B$ satisfies \emph{property} (IV) if
\[ f(x) <1\quad  \text{for every } x\in S_X \text{ and } f\in\B'. \]
Similarly, $\B$ satisfies property (V) if
\[ \sup \{f(x)\colon f\in\B\setminus \D(x)\}<1,\quad  \text{whenever } x\in S_X. \]
By definition, $X$ is (IV)-polyhedral [respectively, (V)-polyhedral] if and only if the set $\ext(B_{X^*})$ satisfies property (IV) [respectively, property (V)].

\begin{lemma}[\cite{FonfLindVes}*{Remark~1.4(a)}, \cite{DEPOLY}*{Lemma~3.4}]\label{lemma: V poly for bdry} Let $X$ be a normed space. Then, the following conditions are equivalent:
\begin{enumerate}
    \item\label{item: extreme} $X$ is \textup{(V)}-polyhedral [respectively, \textup{(IV)}-polyhedral];
    \item\label{item: 1 norming} there exists a $1$-norming set $\B\subseteq B_{X^*}$ satisfying property \textup{(V)} [respectively, property \textup{(IV)}].
\end{enumerate}
\end{lemma}

\begin{proof} By the Krein--Milman theorem, $\ext(B_{X^*})$ is a boundary, thus \ref{item: extreme} implies \ref{item: 1 norming}. We now prove the other implication, and we start with property (V). Let $\B\subseteq S_{X^*}$ be a 1-norming set satisfying $\mathrm{(V)}$, and suppose on the contrary that there exist $x\in S_X$ and $(f_n)_{n=1}^{\infty}\subseteq\ext(B_{X^*})$ such that $f_n(x)\nearrow 1$ as $n\to\infty$ (in particular, $(f_n(x))_{n=1}^{\infty}$ is a strictly increasing sequence). As $\cconv^{w^*}(\B)=B_{X^*}$, Milman's `converse' of the Krein--Milman theorem implies $\ext(B_{X^*})\subseteq \overline\B^{w^*}$. Thus, $(f_n)_{n=1}^{\infty}\subseteq \overline\B^{w^*}$. Consider the $w^*$-open set
\[ U_n\coloneqq \{f\in X^*\colon f_n(x)<f(x)<1\}. \]
Note that $f_{n+1}\in U_n$, so $U_n\cap \overline{\B}^{w^*}\neq \emptyset$. Hence $U_n\cap \B\neq \emptyset$, and there exists $g_n\in \B$ with $f_n(x)<g_n(x)<1$ (in particular, $g_n\notin \D(x)$). Then,
\[ \textstyle \sup\{g(x)\colon g\in\B\setminus\D(x)\}\geq\sup_n g_n(x)=1, \]
which contradicts that $\B$ has property (V).

The proof for property (IV) is simpler. In fact, as before, for each $1$-norming set $\B$ one has $\ext(B_{X^*})\subseteq \overline\B^{w^*}$, thus $(\ext B_{X^*})'\subseteq \B'$. Hence, if $\B$ has property (IV), the same is true for $\ext(B_{X^*})$.
\end{proof}

\begin{lemma}[\cite{DP-Rocky}]\label{lemma: VI iff D(y)} Let $X$ be a normed space, $x\in S_X$, and $V$ be an open neighbourhood of $x$. Then, the following are equivalent:
\begin{enumerate}
    \item\label{i: segment in sphere} $[x,y]\subseteq S_X$, for all $y\in V\cap S_X$;
    \item\label{i: segment after y} for all $y\in V\cap S_X$, $y\neq x$, there exists $p\in V\cap S_X$ such that $[x,p]\subseteq S_X$ and $y\in (x,p)$;
    \item\label{i: D(y) in D(x)} for all $y\in V\cap S_X$ one has $\D(y)\subseteq \D(x)$. 
\end{enumerate}
In particular, $X$ is \textup{(VI)}-polyhedral if and only if for all $x\in S_X$ there is a neighbourhood $V$ of $x$ such that $\D(y)\subseteq \D(x)$ for all $y\in V\cap S_X$.
\end{lemma}

The equivalence between \ref{i: segment in sphere} and \ref{i: D(y) in D(x)} is $(b_1)\iff (c_2)$ in \cite{DP-Rocky}*{Theorem~3}; we give a more direct argument. Condition \ref{i: segment after y} is an intermediate step that we also use in Lemma~\ref{lemma: few extreme points} and it heuristically means that, not only the segment $[x,y]$ is contained in the sphere, but it can also be extended after $y$ without leaving the sphere.

\begin{proof} For \ref{i: segment in sphere}$\implies$\ref{i: segment after y}, take any $y\in V\cap S_X$, $y\neq x$, and a $2$-dimensional subspace $Y$ of $X$ with $\{x,y\}\subseteq Y$. By assumption, $[x,y]\subseteq S_Y$. We show that $y$ is not a vertex of the polygon $B_Y$, which implies \ref{i: segment after y}. In fact, if $y$ were a vertex of $B_Y$, there would be some $p\in S_Y$ with $[p,y]\subseteq S_Y$ and such that the points $x,y,p$ are not aligned. As $V$ is open and $y\in V$, up to replacing $p$ with some point $p'\in [p,y)$ close enough to $y$, we can assume that $p\in V$. Then, the fact that $x,y,p$ are not aligned implies that $[x,p]\not\subseteq S_Y$. As $[x,p]\subseteq Y$, it follows that $[x,p]\not\subseteq S_X$, contradicting \ref{i: segment in sphere}.

Next, for \ref{i: segment after y}$\implies$\ref{i: D(y) in D(x)}, take any $y\in V\cap S_X$. If $y= x$, the inclusion is trivial, so we assume that $y\neq x$. Thus, there is $p\in V\cap S_X$ with $[x,p]\subseteq S_X$ and $y\in (x,p)$. If $f\in \D(y)$, then $f(x)= f(p)= 1$ as well, so $f\in \D(x)$, and \ref{i: D(y) in D(x)} holds.

Finally, for \ref{i: D(y) in D(x)}$\implies$\ref{i: segment in sphere}, take any $y\in V\cap S_X$, and any $z\in [x,y]$. If $f\in \D(y)$, then by assumption $f\in \D(x)$ as well. Thus, $1= f(z)\leq \|z\|\leq 1$, implying that $z\in S_X$. Hence, $[x,y]\subseteq S_X$, which is \ref{i: segment in sphere}.
\end{proof}

\subsection{Some geometric properties}

\begin{lemma}[\cite{InfPoly}*{Theorem~3.6}, \cite{DePolyII}*{Observation 6.9}]\label{lemma: few extreme points} Let $X$ be a polyhedral normed space.
\begin{enumerate}
    \item\label{extreme: Delta} If $X$ is infinite dimensional and has $(\Delta)$, $B_X$ has no extreme points.
    \item\label{extreme: (VI)} If $X$ is \textup{(VI)}-polyhedral, $\ext(B_X)$ is a discrete set.
\end{enumerate}
\end{lemma}

\begin{proof} To prove \ref{extreme: Delta}, suppose that $x\in \ext(B_X)$. By assumption, $\ext(\D(x))$ is a finite set, so $\D(x)= \conv(\ext \D(x))$, by the Krein--Milman theorem. Thus,
\[ \bigcap\{\ker(f)\colon f\in \ext(\D(x)) \}= \bigcap\{\ker(f)\colon f\in \D(x) \}, \] 
which is finite-codimensional. Take $y$ in this intersection so that $Y\coloneqq \spann\{x,y\}$ is $2$-dimensional. As $x\in \ext(B_Y)$, there are two distinct functionals $g,h\in \D_Y(x)$. Note that $g(y)\neq h(y)$, since $g(x)=h(x)=1$ and $\{x,y\}$ is a basis of $Y$. Thus, we assume without loss of generality that $g(y)\neq 0$. However, there is $f\in \D_X(x)$ such that $f\cut_Y= g$, so our choice of $y$ gives $f(y)=0$, a contradiction.

For \ref{extreme: (VI)}, suppose that $x\in \ext(B_X)$, and take an open neighbourhood $V$ of $x$ such that $[x,y]\subseteq S_X$ for all $y\in V\cap S_X$. Then, \ref{i: segment after y} of Lemma~\ref{lemma: VI iff D(y)} holds for $V$, hence no point of $V\cap S_X$ different from $x$ is an extreme point.
\end{proof}

\begin{fact}[\cite{FonfLindVes}*{Observation 1.7}, \cite{InfPoly}*{Proposition 5.2}]\label{fact: DeltaQP iff} A normed space $X$ is a \textup{(VI)}-polyhedral normed space with \textup{($\Delta$)} if and only if every point of $S_X$ is a $\Delta$-QP point for $B_X$.
\end{fact}

\begin{proof} If $X$ is (VI)-polyhedral, by Lemma~\ref{lemma: VI iff D(y)} for each point $x\in S_X$ there is a neighbourhood $U$ of $x$ such that $\D(y)\subseteq \D(x)$ for all $y\in U\cap S_X$. Then, by Fact~\ref{fact: basic D(x)}\ref{i: ext D(x)=}, $B_0\coloneqq \ext B_{X^*}\cap \D(x)= \ext \D(x)$, which is finite by ($\Delta$). Thus, the Krein--Milman theorem implies $\D(x)= \conv (\ext \D(x))$. So, for $y\in U\cap S_X$,
\[ 1= \max_{f\in \D(y)} f(y)\leq \max_{f\in \D(x)} f(y)\leq \max_{f\in B_0} f(y)\leq 1. \]
Hence, for all $y\in \R^+ (U\cap S_X)$, we have that
\[ \|y\|= \max_{f\in B_0} f(y). \]
As $\R^+ (U\cap S_X)$ is a neighbourhood of $x$, this shows that $x$ is a $\Delta$-QP point.

Vice-versa, assume that condition \ref{i: finite true faces} in Lemma~\ref{lemma: Delta-QP pts iff VI+Delta} holds. Then, $X$ is clearly (VI)-polyhedral. Moreover, letting $\B$ be the set of all the functionals inducing the true faces of $B_X$, condition \ref{i: finite true faces} implies that $\B$ is a boundary and $\D(x)\cap \B$ is finite for all $x\in S_X$. By Fact~\ref{fact: Delta iff B cap D finite}, $X$ has ($\Delta$).
\end{proof}

\begin{fact}[\cite{fonf_someproperties}]\label{fact: ctble boundary and IV} Let $(X, \n)$ be a normed space with a countable boundary. Then, $X$ admits an equivalent norm $\nn$ such that $(X,\nn)$ is \textup{(IV)}-polyhedral and admits a countable boundary. Further, the norm $\nn$ can be chosen to approximate $\n$.
\end{fact}

For a Banach space $X$, the result is \cite{fonf_someproperties}*{Theorem~3}; see \cite{InfPoly}*{Theorem~3.3} for a proof. If $X$ is a Banach space, the assertion on the countable boundary is a consequence of the first part, by the structural theorem and the fact that Banach spaces with a countable boundary are separable. On the other hand, the argument in \cite{InfPoly}*{Theorem 3.3} explicitly constructs a countable boundary for the (IV)-polyhedral norm; we briefly recall their proof.

\begin{proof} Suppose that $\{\pm f_n\}_{n=1}^\infty$ is a countable boundary for $X$ and fix a strictly decreasing sequence $(\e_n)_{n=1}^\infty$ that converges to $0$. Define functionals $g_n\coloneqq (1+\e_n) f_n$ and
\[ \vertt{x}\coloneqq \sup_{n\in \N} |g_n(x)|. \]
Take any $x\in X$ and find $k\in \N$ with $\|x\|= |f_k(x)|$; thus, $|f_n(x)|\leq |f_k(x)|$ for all $n\in \N$. Hence, when $n> k$ we get
\[ |g_n(x)|=\frac{1+ \e_n}{1+ \e_k}(1+ \e_k) |f_n(x)|\leq \frac{1+ \e_n}{1+ \e_k}(1+ \e_k) |f_k(x)|= \frac{1+ \e_n}{1+ \e_k} |g_k(x)|. \]
This shows at once that $\{\pm g_n\}_{n=1}^\infty$ is a boundary for $(X,\nn)$ and that it has (IV). By Lemma~\ref{lemma: V poly for bdry}, $(X,\nn)$ is (IV)-polyhedral. The approximation claim is immediate, because by construction $\n\leq \nn\leq (1+\e_1)\n$.
\end{proof}

A particular case when the previous result applies is that of normed spaces with countable dimension. In fact, we have the following result.

\begin{lemma}[\cite{devillefonfhajek}]\label{lemma: ctble dim => IV} Every normed space $X$ of countable dimension admits a \textup{(IV)}-polyhedral equivalent norm (with a countable boundary). Further, the set of \textup{(IV)}-polyhedral norms is dense in the set of all norms.
\end{lemma}

\begin{proof} To begin with, the norm of $X$ can be approximated by a polyhedral norm\footnote{This also follows from \cite{dantashajekrusso}*{Theorem A}.}. In fact, it is proved in \cite{devillefonfhajek}*{Theorem 2.2} that in normed spaces with countable dimension, every bounded convex body can be approximated by polytopes. Thus, $B_X$ can be approximated by a polytope $P$. By symmetry, $-P$ also approximates $B_X$, so $P\cap-P$ is a symmetric polytope that approximates $B_X$. Thus, the polyhedral norm $\n$ whose unit ball is $P\cap-P$ approximates the norm of $X$. Finally, $(X,\n)$ admits a countable boundary by \cite{devillefonfhajek}*{Proposition~2.1}, and Fact~\ref{fact: ctble boundary and IV} applies.

The proof of the fact that polyhedral norms on spaces of countable dimension have a countable boundary is so short that we also recall it here. Write $X= \spann\{e_n\}_{n= 1}^\infty$, and let $X_n\coloneqq \spann\{e_1,\dots, e_n\}$. Then, $X_n$ is a finite-dimensional polyhedral space, so it has a finite boundary $\B_n$. By the Hahn--Banach theorem, we can assume that $\B_n\subseteq B_{X^*}$ (rather than $B_{X_n^*}$); thus $\bigcup_{n= 1}^\infty \B_n$ is a countable boundary for $X$.
\end{proof}

\subsection{Classification of polyhedral normed spaces}
We begin by recalling the classification of polyhedrality notions that we considered in Section~\ref{sec: prelim} and we also give a more detailed reference to the literature.

\begin{theorem} The following relationships among the different notions of polyhedrality hold:
\begin{enumerate}
    \item \label{isometric implications DP} $\textup{(IV)}\implies\textup{(V)}\implies\textup{(VI)}\implies\textup{(K)}$;
    \item \label{isometric implications FV} $\textup{(IV)}\iff\textup{(V)}+(\Delta)$.
\end{enumerate}
\end{theorem}

Item~\ref{isometric implications DP} was proved in \cite{DP-Rocky}*{Theorem 1}, where more properties are also considered. In the statement of \cite{DP-Rocky}*{Theorem 1}, the implications of \ref{isometric implications DP} are (2)$\implies$(4)$\implies$(5)$\implies(7)$. Instead, \ref{isometric implications FV} is stated in \cite{InfPoly}*{Observation~3.5}; the statement given there is for Banach spaces only, but it is easy to check that the argument doesn't require completeness.

Finally, let us comment on some other polyhedrality notions that we didn't investigate in our paper. Consider the following properties of a normed space $X$:
\begin{enumerate}
    \item[(I)] $(\ext B_{X^*})'\subseteq \{0\}$;
    \item[(II)] $(\ext B_{X^*})'\subseteq r B_X$, for some $r<1$;
    \item[(III)] $(\ext B_{X^*})'\subseteq \inte B_X$.
\end{enumerate}

By definition, it is clear that (I)$\implies$(II)$\implies$(III)$\implies$(IV). If $\hat X$ denotes the completion of a normed space $X$, the duals $\hat X^*$ and $X^*$ are canonically isometric; further the $w^*$-topologies induced on them by $\hat X$ and $X$ coincide on bounded sets. As the above properties are only defined in terms of the $w^*$-topology on the dual ball, it follows that for $\textup{j}\in\{ \textup{I, II, III}\}$, a normed space is (j)-polyhedral if and only if its completion is. On the other hand, our results in Section~\ref{sec: counterexamples} give several examples of (IV)-polyhedral normed spaces whose completion is not even (K)-polyhedral. Just to name one, the space $\ell_{1,0}$ admits a (IV)-polyhedral norm by Lemma~\ref{lemma: ctble dim => IV}, but its completion $\ell_1$ doesn't admit any polyhedral norm. Thus, for $\textup{j}\in\{ \textup{I, II, III}\}$ the theory of (j)-polyhedral normed spaces is just obtained by applying the results for complete spaces to the completion. We limit ourselves to a sample result here, which should be compared to Remark~\ref{rmk: bdry of IV}.

\begin{fact}\label{fact: bdry III} If $\hat X$ is polyhedral, $X$ admits a minimal boundary, which has cardinality $\textup{dens}(X)$. In particular, this is the case if $X$ is \textup{(III)}-polyhedral.    
\end{fact}

\begin{proof} This is just a reformulation of Remark~\ref{rmk: true faces from hat X}. In fact, as we noted there, $S_X$ is covered by the true faces obtained intersecting the true faces of $S_{\hat X}$ with $S_X$. In other words, the minimal boundary of $\hat X$ is also the minimal boundary for $X$. Further, this minimal boundary has cardinality $\textup{dens}(\hat X)= \textup{dens}(X)$, by \cite{fonfstruttnew}*{Theorem 1.4}.
\end{proof}

\medskip
\noindent\textbf{Acknowledgements.} A substantial part of this research was carried out during the second author's successive visits to Universit\`{a} Cattolica del Sacro Cuore, Politecnico di Milano, and the University of Innsbruck, whose hospitality and support are gratefully acknowledged.

\medskip
\noindent\textbf{Funding information.} The research of C.A.~De Bernardi has been partially supported by the GNAMPA (INdAM -- Istituto Nazionale di Alta Matematica) and by the MICINN project PID2020-112491GB-I00 (Spain). The research of H. Del R\'{i}o was partially supported by Grant PRE2022-103590, funded by MICIU/AEI/ 10.13039/ 501100011033 and by ESF+, by the grant PID2025-167660NB-I00 funded by MICIU/AEI/10.13039/501100011033 and ERDF/EU, by the "María de Maeztu" Excellence Unit IMAG, funded by MICIU/AEI/10.13039/501100011033 under reference CEX2020-001105-M and by Junta de Andaluc\'ia, grant FQM-0185.  The research of T.~Russo and J.~Somaglia has been partially supported by the GNAMPA (INdAM -- Istituto Nazionale di Alta Matematica). The research of T.~Russo was funded in part by the Austrian Science Fund (FWF) 10.55776/PAT1653825. For open access purposes, the authors have applied a CC BY public copyright license to any author accepted manuscript version arising from this submission.


\end{document}